\documentclass[a4paper, 12pt]{amsart}

\usepackage{mathtools}
\usepackage[bottom=1.5in]{geometry}
\usepackage{fancyhdr}
\usepackage{hyperref}

\theoremstyle{plain}
\newtheorem{theorem}{Theorem}  
\newtheorem{lemma}[theorem]{Lemma}

\newtheorem{proposition}[theorem]{Proposition}

\theoremstyle{definition}

\newtheorem{remark}[theorem]{Remark}

\newcommand{\R}{\mathbb{R}}
\newcommand{\N}{\mathbb{N}}
\newcommand{\argmin}{\mathop{\mathrm{arg\,min}}}
\newcommand{\IT}{\, {\mathop{\mathrm{Im}}} \,}
\newcommand{\Reg}{{\ensuremath{\mathcal{R}}}}

\newcommand{\RT}{\, {\mathop{\mathrm{Re}}} \,}

\newcommand{\supp}{\mathop{\mathrm{supp}}}
\newcommand{\expfct}[1]{e^{#1}}

\def\absnoisy{a^\sigma}

\def\fad{{f_\alpha^\delta}}
\def\fdag{{f^\dag}}

\def\gdag{g^\dag}
\def\grad{\nabla}
\def\IC{\mathbb{C}}
\def\IN{\mathbb{N}}
\def\IR{\mathbb{R}}
\def\knot{\eta}
\def\p{p}

\def\Pim{v}
\def\Pre{u}
\def\psidag{\psi^\dag}
\def\psinoisy{\psi^\delta}

\def\st{~ : ~}
\def\suppk{\ensuremath{\mathfrak{P}}}
\def\sY{s}
\def\Tad{{T_\alpha^\delta}}

\def\tX{\tau}

\def\xad{{x_\alpha^\delta}}

\def\yd{y^\delta}
\def\ydag{y^\dag}

\def\omlow{\omega_{\rm low}}
\def\omup{\omega_{\rm up}}
\def\omcw{\omega_{\rm cw}}

\newcommand{\de}{\,\mathrm d}
\newcommand{\diff}{{\,\mathrm{d}}}

\title[Variational deautoconvolution and phase retrieval]{Variational regularization of complex deautoconvolution and phase retrieval in ultrashort laser pulse characterization}

\author[S.W.~Anzengruber]{Stephan W.~Anzengruber\,}
\address[S.W.~Anzengruber]{J.~Radon Institute for Computational and Applied Mathematics (RICAM), 4040 Linz, Austria, {\rm and}
Faculty of Mathematics, TU Chemnitz, 09107 Chemnitz,  Germany}
\email{\tt stephan.anzengruber@ricam.oeaw.ac.at}
\author[S.~B\"urger]{Steven B\"urger\,}
\address[S.~B\"urger]{Faculty of Mathematics, TU Chemnitz,
  09107 Chemnitz,  Germany}
\email{\tt steven.buerger@mathematik.tu-chemnitz.de}
\author[B.~Hofmann]{Bernd Hofmann\,}
\address[B.~Hofmann]{Faculty of Mathematics, TU Chemnitz,
  09107 Chemnitz,  Germany}
\email{\tt bernd.hofmann@mathematik.tu-chemnitz.de}
\author[G.~Steinmeyer]{G\"unter Steinmeyer}
\address[G.~Steinmeyer]{Max-Born-Institute,  Max-Born-Stra\ss e 2a, 12489 Berlin, Germany}
\email{\tt steinmey@mbi-berlin.de}

\date{\today}

\begin{document}

\begin{abstract}
The SD-SPIDER method for the characterization of ultrashort laser pulses requires the solution of a nonlinear integral equation of autoconvolution type with a device-based kernel function.
Taking into account the analytical background of a variational regularization approach for solving the corresponding ill-posed operator equation formulated in complex-valued
$L^2$-spaces over finite real intervals, we suggest and evaluate numerical procedures using NURBS and the TIGRA method for calculating the regularized solutions in a stable manner.
In this context, besides the complex deautoconvolution problem with noisy but full data, a phase retrieval problem is introduced which adapts to the experimental state of the art in laser optics.
For the treatment of this problem facet, which is formulated as a tensor product operator equation, we derive well-posedness of variational regularization methods.
Case studies with synthetic and real optical data show the capability of the implemented approach as well as its limitation due to measurement deficits.
\end{abstract}

\maketitle

\vspace{1.0cm}

\noindent {\bf MSC2010 subject classification:} 47A52, 47J06, 78A60, 65R32, 45Q05,  65J15

\vspace{0.3cm}


\section{Introduction}\label{s1}

About two decades ago and motivated by a problem from spectroscopy, namely the evaluation of highly resolved functions of the density of unoccupied states from appearance-potential spectra (cf, e.g.,~\cite{Schleietal83}), the inverse problem of deautoconvolution came into the focus of the mathematical literature for the first time (cf.~\cite{Bau91,gh94}). In particular, the stable approximate determination of real functions $f$ over $\R$ with compact support,
say $\supp(f) \subseteq [0,1]$,
from noisy data of its self-convolution $g$ with $\supp(g) \subseteq [0,2]$ became of interest, which is equivalent to the solution of the Volterra type nonlinear integral equation
\begin{equation} \label{eq:real}
 \int\limits_0^\sY f(\sY-\tX)\,f(\tX)\,d\tX \,=\,g(\sY), \qquad 0 \le \sY \le 2,
\end{equation}
occurring also in stochastics if the probability density function $f(\tX)$ with support in $[0,1]$ of a random variable $\mathfrak{X}$ is to be determined from data of the density function $g(\sY)$. In this context, $g(\sY)$
corresponds to the random variable $\mathfrak{Y}+\mathfrak{Z}$, where $\mathfrak{Y}$ and  $\mathfrak{Z}$ are stochastically independent random variables and $\mathfrak{X},\mathfrak{Y},\mathfrak{Z}$ are identically distributed.
It was shown in \cite{gh94} resp.~\cite{FleiHof96} that (\ref{eq:real}) written as an operator equation with the nonlinear forward operator of autoconvolution between the real spaces $L^2(0,1)$ and $L^2(0,1)$ resp.~$L^2(0,2)$ of quadratically integrable functions is locally ill-posed everywhere in the sense of \cite[Definition~2]{HofSch94}. Furthermore, regularization approaches were developed and evaluated, where the astonishing fact appeared that in spite of the simple quadratic structure of the forward operator convergence rates results are difficult to obtain. Indeed, autoconvolution `scrambles' the input function in such a way that qualified nonlinearity conditions like the tangential cone condition and
classical source conditions are generally not satisfied. For details of the corresponding results we refer to \cite{FleiHof96,gh94,Janno97,Janno00} and the recent paper \cite{BuerHof15}. Moreover, we refer to \cite{ChoiLan05, Flem14} for alternative approaches to deautoconvolution.

A completely new approach to the autoconvolution problem from a mathematical point of view was started in 2011 at the TU Chemnitz in collaboration with the research group `Solid State Light Sources' of the Max Born Institute for Non-linear Optics and Short Pulse Spectroscopy, Berlin. The new onset was motivated by the development of the \emph{SD-SPIDER} (self-diffraction spectral phase interferometry for direct electric-field reconstruction) technique in ultrashort (femtosecond) laser pulse characterization at the Max Born institute. A mathematical model for this method can be formulated by generalizing (\ref{eq:real}) as a kernel-based autoconvolution equation
\begin{equation} \label{eq:complex}
\int\limits_0^\sY k(\sY,\tX)\,f(\sY-\tX)\,f(\tX)\,d\tX\,=\,g(\sY), \qquad 0 \le \sY \le 2,
\end{equation}
with complex-valued functions $f$ over $\R$ with $\supp(f) \subseteq [0,1]$ and $g$ over $\R$ with $\supp(g) \subseteq [0,2]$, and a complex-valued kernel $k(\sY,\tX)$ with $(\sY,\tX) \in \R^2$ and $\supp(k) \subseteq [0,2] \times [0,1]$. In the sequel we use the polar coordinate representations
\begin{equation} \label{eq:polar}
f(\tX)= a(\tX)\, \expfct{i\, \varphi(\tX)} \qquad \mbox{and} \qquad g(\sY)=b(\sY)\, \expfct{i\, \psi(\sY)}
\end{equation}
for the searched-for function $f$ and the right-hand side function $g$, respectively.
For a detailed explanation of the physical background, i.e., of the SD-SPIDER  approach in light of (\ref{eq:complex}) and concerning the availability
of optical measurements for the amplitude functions $a,b$ and the phase functions $\varphi,\psi$ in (\ref{eq:polar}), we refer to the subsequent Section~\ref{s2}.

To be precise, we consider (\ref{eq:complex}) as a nonlinear operator equation in a Hilbert space setting. Our focus is on the spaces
\begin{equation} \label{eq:spaces}
X:=L^2_{\IC}(0,1) \quad \mbox{and} \quad Y:=L^2_{\IC}(0,2)
\end{equation}
of complex-valued square integrable functions with associated norms $\|\cdot\|_X,\|\cdot\|_Y $ and inner products $\langle \cdot,\cdot\rangle_X,\langle \cdot,\cdot\rangle_Y$, respectively. Then we write (\ref{eq:complex}) in the concise form
\begin{equation} \label{eq:opeq}
F(f)\,=\,g,\qquad f \in X,\quad g \in Y,
\end{equation}
where $\fdag \in X$ denotes exact solutions to (\ref{eq:opeq}) for given exact right-hand side $g=\gdag \in Y$, and where the forward operator $F: X \to  Y$, taking into account the specific support intervals of pre-image and image functions, is defined as
  \begin{equation} \label{eq:F}
    \begin{aligned}
              [F(f)] (\sY) & = [B(f)](\sY)\,\expfct{i\, [\Psi(f)](\sY)} \\
                     & := \int \limits_{\max(\sY-1,0)}^{\min(\sY,1)} k(\sY,\tX)\, f(\sY-\tX)\, f(\tX)\, d\tX, \qquad 0 \le \sY \le 2,
    \end{aligned}
  \end{equation}
for preimage function $f(\tX)=a(\tX)\,\expfct{i\, \varphi(\tX)}$, $0 \le \tX \le 1$.
From (\ref{eq:F}) it is seen that the support of the kernel function $k$ is contained in the parallelogram
  \begin{equation} \label{eq:suppk}
\suppk:= \{ (\sY, \tX) \in [0,2] \times [0,1] \st ~ 0 \leq \tX \leq 1, ~ \tX \leq \sY \leq \tX+1 \}.
  \end{equation}
According to physical models (cf. \cite{BSKGBH15, Ge11, GHBKS13}) the kernel functions occurring in laser pulse characterization are continuous complex-valued functions on $\suppk$. We will thus restrict our considerations to  kernel functions, such that
  \begin{equation} \label{eq:kernel}
    \supp(k) \subset \suppk, \quad k \in C (\suppk), \quad \mbox{and } \quad k(\sY,\tX)=k(\sY,\sY-\tX) \mbox{ for } (\sY, \tX) \in \suppk.
  \end{equation}
The last assumption in (\ref{eq:kernel}) indeed holds without loss of generality as we have,  for arbitrary kernel functions $k(\sY,\tX)$ and all $\sY \in [0,2]$, the identity
$$ \int \limits_{\max(\sY-1,0)}^{\min(\sY,1)} k(\sY,\tX) f(\sY-\tX) f(\tX)\, d\tX=\int \limits_{\max(\sY-1,0)}^{\min(\sY,1)} \frac{k(\sY,\tX)+k(\sY,\sY-\tX)}{2} f(\sY-\tX) f(\tX) d\tX.$$
We mention that the structure (\ref{eq:kernel}) is in particular satisfied if the kernel is generated by a complex-valued function $\kappa$ with $\supp(\kappa) \subset [0,1]$ and $\kappa \in C [0,1]$ such that
\begin{equation} \label{eq:simplekernel}
  k(\sY,\tX)\,=\, \kappa(\tX) \, \kappa(\sY-\tX) \quad \mbox{if } (\tX,\sY) \in \suppk, \qquad \mbox{and} \quad k(\sY,\tX)=0 \quad \mbox{otherwise}.
\end{equation}
For some statements in Section~\ref{s3}, we will be forced to focus on the special case (\ref{eq:simplekernel}).

In this paper, we distinguish two different inverse problems in the context of the forward operator $F$ from (\ref{eq:F}) under the assumption (\ref{eq:kernel}) for the kernel $k$, which are briefly presented in the following.

\medskip

\subsection{The deautoconvolution problem}   \label{ssec:deautoconv}
We call the inverse problem of identifying the complex-valued function $f$ solving the operator equation (\ref{eq:opeq}) deautoconvolution problem. In this context, it is assumed that only noisy data $g^\delta$ of the complex-valued function $\gdag$ with
  \begin{equation} \label{eq:noise}
    \|\gdag-g^\delta\|_{Y} \le \delta
  \end{equation}
are available for a reasonably small noise level $\delta>0$.

\medskip

 This deautoconvolution problem was tackled in \cite{GHBKS13} in a direct manner, where theoretic consideration were only made for the trivial kernel $k \equiv 1$.  For the stable approximate solution including a
 nontrivial kernel, an iterative regularization approach based on a variant of the Levenberg-Marquardt method was used there, which showed quite good results in numerical case studies with synthetic data. In this paper,
 we extend assertions on properties of the forward operator $F$ and on the solution of the operator equation (\ref{eq:opeq}) to kernels $k$ from (\ref{eq:kernel}) or (\ref{eq:simplekernel}).

\smallskip

The real data situation in laser optics, see for details Section~\ref{s2} below, has an advantage and a disadvantage. Fortunately, noisy measurements for the amplitude function $a$ in the solution $f$ (see (\ref{eq:polar})) can be provided. These data were exploited in \cite{GHBKS13} to control a regularization parameter. However, measurements for the amplitude function $b$ in $g$, which were required in combination with data for the phase $\psi$, are unfortunately not sufficiently reliable in practice. As a result, the method in \cite{GHBKS13} failed for real data from optical experiments.
In order to resolve this shortcoming, we clarify the objective and the data situation in the following manner.

\medskip

\subsection{The phase retrieval problem}   \label{ssec:phaseretrieval}
We call the inverse problem of identifying the phase function $\varphi$ in the solution $f$ (see (\ref{eq:polar})) of equation (\ref{eq:opeq}) from noisy data of the phase function $\psi$ in the right-hand side $g$ phase retrieval problem when noisy data of the amplitude function $a$ in $f$, but no data of the amplitude function $b$ in $g$ are available.

\medskip

A first trial of a more adapted regularization approach with focus on the phase retrieval problem was included in the recent paper \cite{BSKGBH15} taking into account the real data situation. However, the well-posedness of the non-standard variational regularization methods was not considered there. We are going to
close this gap in the present paper, in addition to an essential improvement of the numerical implementation based on \emph{NURBS} (Non-Uniform Rational B-splines; see, e.g.,~\cite{pt97}) in combination with a \emph{TIGRA}-type algorithm (the name is derived from TIkhonov-GRAdient method; cf.~\cite{Ramlau02,Ramlau03,WARH15}).

\medskip

The paper is organized as follows. After an explanation of the background from laser optics in Section~\ref{s2} we summarize properties of the forward operator $F$ as well as of the deautoconvolution problem solving equation (\ref{eq:opeq}) in Section~\ref{s3}. Well-posedness results for the phase retrieval problem are presented in Section~\ref{s4}.
Turning to numerical considerations, we then briefly introduce planar NURBS and describe how they may be utilized in the context of the complex-valued deautoconvolution problems in Section \ref{sec:NURBS}.
Finally, in Section \ref{sec:numerics} we show numerical results obtained with a TIGRA-type method for both synthetic and real data.

\section{Physical Background}\label{s2}

The motivation for the complex-valued and kernel-based deconvolution problem of solving (\ref{eq:opeq}) in the introduction is the SD-SPIDER (Self-Diffraction Spectral Phase Interferometry for Direct Electric-field Reconstruction) method in laser optics.
The aim of this method is the reconstruction of the electric field $E(t)$ of ultrashort (femtosecond) laser pulses, which is a real-valued oscillatory function of time $t$. This function is usually decomposed into
an amplitude part and an oscillating part by
\begin{equation}
 E(t)={\sqrt{I(t)}}\cos\{\omega_0 t+{\eta(t)}\},
\end{equation}
where $I(t)$ is the intensity (up to re-normalization), $\omega_0$ the carrier frequency, and $\eta(t)$ is called temporal phase. Since measurements are available for the spectral domain only,
we have to consider the Fourier transform of $E(t)$
\begin{equation}
 \hat E(\omega):=\frac1{\sqrt{2\pi}}\int\limits_{-\infty}^\infty E(t) \expfct{-i\omega t} \mathrm dt
\end{equation}
as a complex-valued function of the frequency $\omega$.
The Fourier transform can be written in polar coordinates as
\begin{equation}
 \hat E(\omega)={\sqrt{S(\omega)}} \expfct{i\phi(\omega)},
\end{equation}
where $S(\omega)$ is the spectral power density (up to re-normalization), and $\phi(\omega)$ is called spectral phase. Fortunately, one is able to measure the spectrum directly, but contaminated with noise. Hence, the approximate identification of the desired physical quantity $E(t)$ by inverse Fourier transform requires first the determination of $\phi$. Careful denoising, e.g., by Fourier filtering or adjacent averaging of the measurements for $S$ may prove helpful to reduce a possible oscillatory behavior of the algorithms. Since the spectral phase $\phi$ cannot be measured directly, nonlinear optical processes have to be employed to infer the spectral phase from an indirect measurement. To this end, one suitable process is self-diffraction (SD), which is a spectrally degenerate variant of the more general four-wave mixing process.
The electric field of the generated pulses in the spectral domain $\hat E_{\rm SD}(\omega)$ is related to $\hat E(\omega)$ through
\begin{equation}\label{eq1}
 \hat E_{\rm SD}(\omega)=\int\limits_{0}^{\omega+\omcw}K(\omega,\Omega)\hat E(\omega+\omcw-\Omega)\hat E(\Omega)\mathrm d \Omega.
\end{equation}
The SD process involves interaction with a continuous wave at a known frequency $\omcw$, which has been incorporated into the kernel function $K(\omega,\Omega)$.
Moreover, the kernel function $K(\omega,\Omega)$ as an apparatus function is in principle known from physical modeling and can be assumed to be smooth in both arguments $\omega$ and $\Omega$.
For a detailed description of the kernel see \cite{Ge11}.
We also write $\hat E_{\rm SD}$ in polar coordinates as
\begin{equation}
 \hat E_{\rm SD}(\omega)=\sqrt{S_{\rm SD}(\omega)} \expfct{i\phi_{\rm SD}(\omega)}.
\end{equation}
Perfect knowledge of all efficiency calibration factors in $K$ provided, one can directly use the spectral power density $S_{\rm SD}$ for reconstruction of $E(t)$. Even a relatively small miscalibration, however, may have a dramatic influence on $S_{\rm SD}$, 
which is why algorithms relying on the spectral power density may have trouble to produce reasonable results \cite{Ge11}.
The phase $\phi_{\rm SD}$, on the other hand, remains widely unaffected from such amplitude calibration problems.
The measurement setup for the SD-interferogram is shown in Figure \ref{fig:physics}. For physical details we refer to \cite{Ge11,GHBKS13,KokBirBetGreSte10}.

  \begin{figure}[htb]
    \begin{center}
      \includegraphics[width=0.8\textwidth]{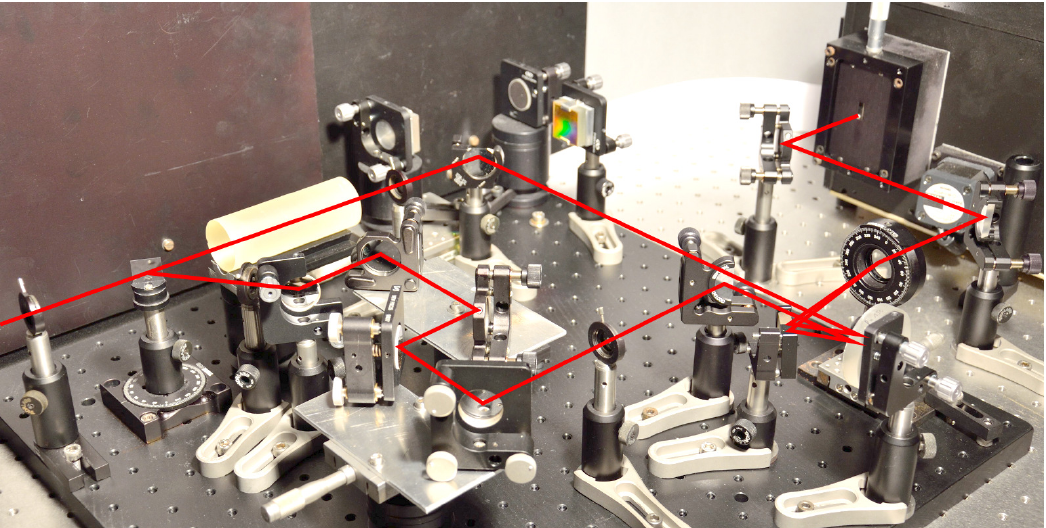}
      \caption{Measurement setup in self-diffraction spectral interferometry.} \label{fig:physics}
    \end{center}
  \end{figure}

The measurements for the spectrum of $\hat E$ indicate that the corresponding function can be neglected outside a compact interval $[\omlow, \omup]$, i.e.~we assume that
\begin{equation}\label{supp1}
\text{supp } \hat E\subset [\omlow,\omup].
\end{equation}
 Consequently, $\hat E_{\rm SD}$ can also be neglected outside the compact interval
\begin{equation}\label{supp2}
 \text{supp }\hat E_{\rm SD}\subset [2\omlow-\omcw,2\omup-\omcw].
\end{equation}
Thus, (\ref{eq1}) can be written as
\begin{equation}\label{eq2}
\hat E_{\rm SD}(\omega)=\int\limits_{\max\{\omlow,\omega+\omcw-\omup\}}^{\min\{\omup,\omega+\omcw-\omlow\}}K(\omega,\Omega)\hat E(\omega+\omcw-\Omega)\hat E(\Omega)\mathrm d \Omega.
\end{equation}
The substitutions
\begin{align*}
  g(\sY):=& \hat E_{\rm SD}(-\omcw+2\omlow+\sY(\omup-\omlow)),\\
  f(\tX):=& \hat E(\omlow+\tX(\omup-\omlow)),\\
  k(\sY,\tX):=&K(-\omcw+2\omlow+ \sY(\omup-\omlow),\omlow+\tX(\omup-\omlow))
\end{align*}
 reformulate (\ref{eq2}) as 
\begin{equation}\label{eq3}
 \int\limits_{\max\{0,\sY-1\}}^{\min\{1,\sY\}}k(\sY,\tX)\,f(\sY-\tX)\,f(\tX)\mathrm d\sY = g(\sY)\,,\qquad 0 \le \sY \le 2,
\end{equation}
which is precisely the mathematical model presented in the introduction. The physical background justifies the continuity of the kernel $k$ in $\suppk$ as required in (\ref{eq:kernel}).

\section{Some properties of complex autoconvolution operator and deautoconvolution problem}\label{s3}

In this section, we summarize properties of the complex autoconvolution operator and discuss the classical Tikhonov regularization for the deautoconvolution problem aimed at solving the operator equation (\ref{eq:opeq})
based on noisy data $g^\delta$ under the noise model (\ref{eq:noise}).

  The Fr\'echet differentiability of the autoconvolution operator mapping in the real Hilbert space $L^2(0,1)$ and the structure of its Fr\'echet derivative were outlined in \cite{gh94}. Under the kernel assumption (\ref{eq:kernel}) imposed on the kernel $k$, such result can also be formulated for the complex case (\ref{eq:F}) with the Fr\'echet derivative \linebreak
  $F'(f): X=L^2_{\IC}(0,1) \to Y=L^2_{\IC}(0,2)$ given for all $f \in X$ by the formula
  \begin{equation} \label{eq:Frech}
  [F'(f)h](\sY)  = 2 \int \limits _{\max(\sY-1,0)}^{\min(\sY,1)} k(\sY, \tX)\, f(\sY-\tX)\, h(\tX)\, d\tX, \qquad 0 \le \sY \le 2, \quad h \in X,
  \end{equation}
  from which we easily derive
  \begin{equation} \label{eq:equa}
[F(f+h) - F(f) - F'(f)\,h] (\sY) =  [F(h)](\sY) , \qquad 0 \leq \sY \leq 2,\quad h \in X.
\end{equation}
  Hence, we obtain with $\suppk$ from (\ref{eq:suppk}) and using
  \begin{equation} \label{eq:upk}
    \bar k:= \max \limits_{(\sY,\tX) \in \suppk} |k(\sY,\tX)|,
  \end{equation}
  the norm equation
 $$  \|F(f+h) - F(f) - F'(f)\,h\|_Y=\|F(h)\|_Y \quad \mbox{for all}\;\;f,h \in X$$
  and the nonlinearity condition
    \begin{equation}\label{eq:Lipschitz}
  \|F(f+h) - F(f) - F'(f)\,h\|_Y \le \bar k \,\|h\|_X^2 \quad \mbox{for all}\;\;f,h \in X.
  \end{equation}
    Such condition, which was already used in~\cite[Section~10.2]{EHN96} to obtain convergence rates for Tikhonov regularization of nonlinear operator equations, is by now the only available nonlinearity condition for the operator $F$ from (\ref{eq:F}).

  Note that the adjoint operator $F'(f)^*: Y \to X$ can also be written down explicitly, namely as
  \begin{equation} \label{eq:Frech*}
    [F'(f)^*r](\tX)  = 2 \int \limits_\tX^{\tX+1} k(\sY,\tX)\, f(\sY-\tX)\, r(\sY)\, d\sY, \qquad 0 \le \tX \le 1, \quad f \in X,\; r \in Y.
  \end{equation}
  As usual we apply in the sequel the symbol $\mathcal{B}(f,r)$ for the closed ball in $X$ with center $f$ and radius $r>0$ and the symbols $\rightharpoonup$  and $\to$ for weak and norm convergence, respectively, in the occurring Hilbert spaces. The proofs of the following two results have been postponed to the Appendix. 

\begin{proposition} \label{pro:weaklycont}
 Suppose that the kernel $k$ satisfies (\ref{eq:kernel}). Then the autoconvolution operator $F:X \to Y$ from (\ref{eq:F}) is weakly continuous, i.e., for every sequence $\{f_n\}_{n \in \N} \in X$ with $f_n \rightharpoonup f_0$ in $X$ as $n \to \infty$ we have $F(f_n) \rightharpoonup F(f_0)$ in $Y$.
\end{proposition}

  The following proposition outlines a very specific property of the autoconvolution operator, namely that the nonlinear $F: L^2_{\IC}(0,1) \to L^2_{\IC}(0,2)$ from (\ref{eq:F}) is non-compact, whereas the linear Fr\'echet derivative operator $F^\prime(f)$ is compact in all points $f \in L^2_{\IC}(0,1)$. This is a remarkable property of the autoconvolution operator, which rarely occurs for nonlinear operators. Conversely, it is well-known that  the Fr\'echet derivative of a compact nonlinear operator is always compact. Again the proof can be found in the Appendix. 

\begin{proposition} \label{pro:noncompact}
 Suppose that the kernel $k \not \equiv 0$ satisfies (\ref{eq:kernel}).
 Then the kernel-based nonlinear autoconvolution operator $F: X \to Y$ from (\ref{eq:F}) is not compact.
 More precisely, there exists a sequence $\{h_n\}_{n \in \N} \subset \mathcal{B}(0,1)$ such that $h_n \rightharpoonup 0$ but $h_n \not \to 0$ in $X$ as $n \to \infty$, which implies that we have, for all $f \in X$  and all $r>0$,  a sequence $\{f_n:=f+r\,h_n\}_{n \in \N} \subset \mathcal{B}(f,r)$ with $f_n \rightharpoonup f$  in $X$, $F(f_n) \rightharpoonup F(f)$  in $Y$ but $F(f_n) \not \to F(f) $ in $Y$ as $n \to \infty$.
 On the other hand, the Fr\'echet derivative $F^\prime(f)$ given by (\ref{eq:Frech}) is a compact linear operator for all $f \in X$.
 \end{proposition}

\medskip

\begin{remark} \label{rem:ill-posed}
  The proof of Proposition \ref{pro:noncompact} is based on finding sequences $f_n \rightharpoonup f$ in $X$ such that $F(f_n) \not \to F(f)$ in $Y$ as $n \to \infty$.
  A conversely related concept  is the 
  local ill-posedness of the operator equation (\ref{eq:opeq}) at $f \in X$, which requires sequences $f_n$ in any neighborhood $\mathcal{B}(f,r)$ of $f$ such that $f_n \not \to f$, but $F(f_n) \to F(f)$. For the specific case $k\equiv 1$ on $\suppk$, an example of the form
    \[ f_n = f + r\,h_n, \qquad \mbox{with } h_n(\tX) = \expfct{i\,n^2\,\tX^2}, \]
  was provided in \cite[Example 3.2]{BuerHof15}, showing that the  operator equation (\ref{eq:opeq}) is locally ill-posed everywhere in this situation. It is worth noting that, while the freedom in choosing $r>0$ is exploited here for local arguments with arbitrarily small $r$, it may also be used for construction of elements $f_n$ at arbitrarily large distances, \mbox{$\|f-f_n\|_X = r$}, whose images $F(f)$ and $F(f_n)$ are virtually indistinguishable, say \mbox{$\| F(f)-F(f_n)\|_Y \leq \frac 1n$}. This means that highly oscillating perturbations imposed on $f$ cannot be recovered by a simple least-squares approach from noisy data $g^\delta$ of $g=F(f)$ with noise model (\ref{eq:noise}), even if the noise level $\delta>0$ is arbitrarily small. Since the same phenomenon has to be expected for more general kernels $k$, a regularization method seems to be always required in order to avoid oscillating numerical approximations.
\end{remark}

  These observations are particularly interesting in combination with the uniqueness assertion in Proposition~\ref{prop:unique} below which is based on the following well-known Titchmarsh convolution theorem; cf.~\cite{Tit26}.

 \begin{lemma} \label{lem:Titch}
Let $\xi_1, \xi_2 \in L^1_{\IC}(\R)$ with $\supp(\xi_l) \subset [0,\infty),\; l=1,2$, and let for some constant $a>0$
$$ \int \limits_0^\sY \xi_1(\sY-\tX)\,\xi_2(\tX)\,d\tX =0 \quad \mbox{for almost all} \quad \sY \in [0,a].$$
Then there are nonnegative constants $a_1$ and $a_2$ such that $a_1+a_2 \ge a$ and
$$\xi_1(\tX)=0 \;\; \mbox{for almost all} \;\; \tX \in [0,a_1]\;\; \mbox{and} \;\; \xi_2(\tX)=0  \;\; \mbox{for almost all} \;\; \tX \in [0,a_2]. $$
 \end{lemma}

\begin{proposition} \label{prop:unique}
 Let the kernel $k$ from (\ref{eq:kernel}) be generated by a function $\kappa \in C[0,1]$ with $\supp(\kappa) = [0,1]$ such that (\ref{eq:simplekernel}) is satisfied.
 If for given $g \in Y = L_{\IC}^2(0,2)$ the function $\fdag \in X = L_{\IC}^2 (0,1)$ solves (\ref{eq:opeq}) then $\fdag$ and $-\fdag$ are the only solutions of this equation for the right-hand side $g$.
 \end{proposition}

 \begin{proof}
Based on Lemma~\ref{lem:Titch}, the assertion of this proposition can be shown in analogy to the proof of Theorem~4.2 in \cite{GHBKS13}, taking into account that under the stated assumptions on $\kappa$ the equality
$\kappa(\tX)\,\fdag(\tX)=\kappa(\tX)\,\tilde f(\tX)$ with $\tilde f \in X$  and for almost all $\tX \in [0,1]$ implies that $\fdag$ and $\tilde f$ are the same elements of $L^2_{\IC}(0,1)$ when taking into account that
    \[ \supp(\kappa) = [0,1] \quad \mbox{holds\, if and only if} \quad \kappa \neq 0 \mbox{ almost everywhere on } [0,1]. \]
 \end{proof}

Note that Lemma~\ref{lem:Titch} under the assumption (\ref{eq:kernel}) also provides us with a necessary and sufficient condition for the injectivity of the Fr\'echet derivative; see~(\ref{eq:Frech}).
Namely, the bounded linear operator $F^\prime(f): X \to Y$ is non-injective if either $f$ is the zero function almost everywhere on $[0,1]$ or if $\kappa \equiv 0$. Otherwise the Fr\'echet derivative is injective.

\bigskip

Variational (cf.~\cite{Scherzetal09}) and iterative (cf.~\cite{KNS08}) regularization methods form two classes of standard methods for stabilizing ill-posed nonlinear operator equations (\ref{eq:opeq}) in Hilbert spaces.
The most prominent representative for the first class is the Tikhonov regularization, where in the simplest form (cf., e.g., \cite[Chapt.10]{EHN96}) regularized solutions $\fad$ are minimizers of the extremal problem
\begin{equation} \label{eq:EHNTik}
    \|F(f)-g^\delta\|_Y^2+ \alpha\,\|f-\bar f\|_X^2 \to \min,\quad
    \mbox{subject to} \quad f \in X,
\end{equation}
with some initial guess $\bar f \in X$.
For obtaining convergence rates of the regularized solutions, an appropriate interplay of solution smoothness and structural conditions expressing the nonlinearity of $F$ in a neighborhood of the solution is required (see, e.g., \cite{Hof15} for an overview).
For the Tikhonov regularization of the form (\ref{eq:EHNTik}) and  $F$ from (\ref{eq:F}), the condition (\ref{eq:Lipschitz}) acts as nonlinearity condition sufficiently well and allows proving the convergence rate
$$\|f^\delta_{\alpha(\delta)}-\fdag\|_X=\mathcal{O}\left(\sqrt{\delta} \right) \quad \mbox{as} \quad \delta \to 0 \quad \mbox{and} \quad  \alpha(\delta)\sim \delta,  $$
since the operator $F$ is weakly continuous (cf.~Proposition~\ref{pro:weaklycont}), which implies that $F$ is weakly closed. The latter result additionally requires the existence of a source element $w \in Y$  satisfying the smallness condition $\|w\|_Y \le 1$ such that the source condition
\begin{equation} \label{eq:source}
\fdag(\tau)-\bar f(\tau)= \int \limits _\tau^{\tau+1} k(\sY, \tX)\, \fdag(\sY-\tX)\, w(\sY)\, d\sY ,\qquad 0 \le \tau \le 1,
 \end{equation}
is fulfilled (cf.~\cite[Theorem~10.4]{EHN96}). However, in \cite[Proposition~2.6]{BuerHof15} it was shown that such source condition (\ref{eq:source}) is hardly possible to achieve even in the simplest cases of a kernel $k \equiv 1$.

\medskip

For wide classes of iterative regularization methods, however, the tangential cone condition
 \begin{equation}\label{eq:TCC}
  \|F(f+h) - F(f) - F'(f)\,h\|_Y \le \bar c\, \|F(f+h) - F(f)\|_Y\quad \mbox{for all} \;h \in \mathcal{B}(f,r)
  \end{equation}
is required, with a constant $0 \le \bar c <1,$ some radius $r>0$, and at least for $f \in X$ in a neighborhood of the solution $\fdag$ to (\ref{eq:opeq}) (see, e.g., \cite[Chapt.11]{EHN96}). Slightly modified variants of the nonlinearity condition (\ref{eq:TCC}) for
constants  $0 \le \bar c< \infty$ and terms $\theta(\|F(f+h) - F(f)\|_Y)$ instead of $\|F(f+h) - F(f)\|_Y$, with concave strictly increasing functions $\theta:(0,\infty) \to (0,\infty)$ and
 $\lim \limits_{t \to +0} \theta(t)=0$, are also relevant for obtaining convergence rates in Tikhonov regularization, in particular if the solution smoothness is low and approximate source conditions apply (see \cite[Section~4.2]{SKHK12}).

\begin{proposition} \label{pro:noTCC}
For the autoconvolution operator $F:X \to Y$ from (\ref{eq:F}), the tangential cone condition (\ref{eq:TCC}) cannot hold for any $f \in X$ with a constant $0 \le \bar c <1$ and a radius $r>0$.
\end{proposition}
\begin{proof}
The assertion is, in principle, a consequence of the noncompactness of $F$ in contrast to the compactness of $F^\prime(f)$ for all $f \in X$. Using the triangle inequality
$$ \|F(f+h) - F(f)\|_Y \le  \|F(f+h) - F(f) - F'(f)\,h\|_Y + \|F'(f)\,h\|_Y  ,$$
we would have from (\ref{eq:TCC}), for fixed $f \in X$ and $r>0$, the inequality
$$(1-\bar c) \|F(f+h) - F(f)\|_Y \le  \|F'(f)\,h\|_Y \quad \mbox{for all} \;h \in \mathcal{B}(f,r).$$
By substituting $h:=h_n$ in that inequality with $\{h_n\}_{n \in \N} \subset \mathcal{B}(f,r)$ from Proposition~\ref{pro:noncompact} satisfying $h_n \rightharpoonup 0, \; h_n \not \to 0$ in $X$ and $F(f+h_n) \not \to F(f)$ in $Y$ as $n \to \infty$, we arrive at a contradiction, because $F^\prime(f)$ is compact and hence satisfies the limit condition $\lim \limits_{n \to \infty}\|F^\prime(f)\,h_n \|_Y= 0$. This contradiction proves the proposition.
\end{proof}

\begin{remark} \label{rem:propend}
As originally discussed in \cite{HofSch94}, the four norm terms associated with the Taylor remainder equation, $$F(f+h)-F(f)=F^\prime(f)\,h+{\rm Rem}(f,h),$$ namely
$\|F(f+h)-F(f)\|_Y$, $\|{\rm Rem}(f,h)\|_Y=\|F(f+h)-F(f)-F^\prime(f)\,h\|_Y$,  $\|F^\prime(f)\,h\|_Y$, and $\|h\|_X$ show distinguished cross connections depending on whether the corresponding
nonlinear operator equation (\ref{eq:opeq}) is well-posed or ill-posed.
The convergence of well-posed problems usually results from the fact that the remainder $\|{\rm Rem}(f,h)\|_Y$ converges to zero faster
than the term $\|F(f+h)-F(f)\|_Y$ as \mbox{$\|h\|_X \to 0$}. In the ill-posed situation, however, where $F$ is a `smoothing' operator, the term $\|F(f+h)-F(f)\|_Y$ may be significantly smaller than $\|{\rm Rem}(f,h)\|_Y$ even for arbitrarily small $\|h\|_X$.
In the latter case, there exist operators $F$ for which the tangential cone condition (\ref{eq:TCC}) may fail to hold even for large constants $\bar c \ge 1$ as well as for modifications with terms $\theta(\|F(f+h) - F(f)\|_Y)$ rather than $\|F(f+h) - F(f)\|_Y$.
Then only Lipschitz continuity conditions for $F^\prime(f)$ such as (\ref{eq:Lipschitz}) can be seen as nonlinearity conditions. In this particular situation, the linear operator $F^\prime(f)$ does not cover sufficient information about the nonlinear operator in a neighborhood of $\mathcal{B}(f,r)$ for ensuring convergence rates of regularized solutions.
This information deficit seems to be the case for the autoconvolution operator $F$ from (\ref{eq:F}) and, unfortunately, prevents the establishment of convergence rates in Tikhonov regularization when the standard source condition
(\ref{eq:source}) fails.
\end{remark}

\section{A regularization approach for the phase retrieval problem}\label{s4}

\def\wto{\rightharpoonup}

   In this section, we collect some basic well-posedness results concerning the regularization approach by means of minimizing Tikhonov-type variational functionals for the different problems summarized in Subsections \ref{ssec:deautoconv} and \ref{ssec:phaseretrieval}.

   The theory of Tikhonov-regularization is well understood for ill-posed problems of the general form
    \begin{equation} \label{eq:invprob}
      \mathcal{F} (x) = y,
    \end{equation}
   where $x$ is a searched-for quantity of interest and $y$ is approximately known from measurements in the form of data $\yd$.
   We refer to the monographs \cite{EHN96, Scherzetal09, SKHK12} as well as to the seminal works \cite{EKN89, HKPS07} for detailed regularization results.
   For our purposes, let us recall the main sufficient condition on $\mathcal{F}: \mathrm{dom} (\mathcal{F}) \subset \mathcal{X} \to \mathcal{Y}$ \cite[Section 3.2]{Scherzetal09} for the well-posedness of minimizing a Tikhonov-type functional
    \[ \mathcal{T}_\alpha^\delta (x) = \frac 12 \big \| \mathcal{F} (x) - \yd \big \|_Y^2 + \alpha \Reg (x) \]
   in Hilbert spaces $\mathcal{X}$, $\mathcal{Y}$: namely, that the forward operator $\mathcal{F}$ be weakly sequentially closed in the sense that
      \[ x_n \wto \bar x \in \mathcal{X}  \mbox{ and }  \mathcal{F} (x_n) \wto \bar y \in \mathcal{Y}  \quad \Rightarrow \quad  \bar x \in \mathrm{dom} (\mathcal{F}) \mbox{ and } \mathcal{F} (\bar x) = \bar y \]
   hold for all sequences $\{ x_n \} \subset \mathrm{dom} (\mathcal{F})$.

   \begin{remark} \label{rem:well-posed}
      Note that we do not specify further the choice of the penalty term $\Reg (x)$. For the following results to hold true,  $\Reg (x)$ is required to be proper, weakly sequentially lower semicontinuous and to have weakly sequentially precompact sublevelsets. We will summarize these properties by saying that  $\Reg (x)$ is assumed to be {\em stabilizing}. The interested reader will easily verify that for stabilizing penalty terms $\Reg (x)$ in combination with weakly sequentially closed forward operator $\mathcal{F}$ with $\mathrm{dom} (\mathcal{F}) = \mathcal{X}$, Assumption 3.13 in \cite{Scherzetal09} is satisfied (with the exception of convexity of $\Reg$ which is, however, not required here). Consequently, the well-posedness results in \cite[Section~3.2]{Scherzetal09} hold true. In particular, we have:
      \begin{itemize}
       \item Existence of minimizers (cf.~Theorem~3.22);
       \item Weak, subsequential stability of the minimizers  (cf.~Theorem~3.23);
       \item Weak, subsequential convergence of the minimizers to an $\Reg$-minimizing solution as $\delta \to 0$ under suitable parameter choice rules (cf.~Theorem~3.26).
      \end{itemize}
   \end{remark}

   The deautoconvolution problem of Subsection \ref{ssec:deautoconv} is naturally of the form (\ref{eq:invprob}) with $\mathcal{F} = F$ from (\ref{eq:F}), and due to the properties in Section \ref{s3} we immediately obtain the following regularization result.

  \begin{proposition}
    Suppose that the kernel $k$ satisfies (\ref{eq:kernel}) and that the penalty term $\Reg (f)$ is {\em stabilizing} in $X = L^2_{\IC}(0,1)$. Then for any $g^\delta \in Y = L^2_{\IC}(0,2)$, minimizing the Tikhonov functional
      \[ \mathcal{T}_\alpha^\delta (f) = \frac 12 \big \| F (f) - g^\delta \big \|_Y^2 + \alpha \Reg (f) \]
    for $f \in X$ is well-posed in the sense of Remark \ref{rem:well-posed}.
  \end{proposition}

  \begin{proof}
    The weak sequential closedness of $F: X \to Y$ follows readily from the weak sequential continuity in Proposition \ref{pro:weaklycont}. The result then follows arguing as in Remark \ref{rem:well-posed}.
  \end{proof}

  In the particular situation of phase retrieval problems as in Subsection \ref{ssec:phaseretrieval}, the data consists of two separate parts. On the one hand $a^\sigma$ contains information about the absolute value of the unknown function $f$, and, on the other hand, $\psinoisy \approx \arg (y^\dag)$ is related to the image space $Y$. For formulating this problem in the context of Tikhonov regularization, we first define the $\mathrm{Sign}$ operator on $L^2_{\IC}(0,2)$ pointwise almost everywhere (a.e.) by
      \[ [\mathrm{Sign} (g)] (\sY)  = \left\{ \begin{array}{ll} 	\frac {g(\sY)}{|g(\sY)|}	& \mbox{ if } g(\sY) \neq 0 \\[2mm]
									0				& \mbox{ else. } \end{array}\right. \]
  Then the forward operator in the sense of (\ref{eq:invprob}) also consists of two parts and maps as
    \begin{align*}
      \mathcal{F}_{\mathrm{PR}} : X = L^2_{\IC}(0,1) & \to \mathcal{Y} = L^2_{\IC}(0,2)  \times L^2(0,1)\\
      \mathcal{F}_{\mathrm{PR}} (f) & = \big( \mathrm{Sign} ( F(f) ), ~ |f| \big),
    \end{align*}
  where the function $|f| \in L^2(0,1)$ is again defined pointwise a.e.\,by
    $ |f|(\tX) = |f(\tX)| $
  for $f \in X$. Note that the product space $\mathcal{Y} = Y  \times L^2(0,1)$ is a Hilbert space endowed with any of the equivalent norms $\|(y,a)\|_{\mathcal{Y},\beta}^2 := \|y\|_Y^2 + \beta \|a\|_{X}^2$ corresponding to discrepancy terms
      \[ \big\| \mathcal{F}_{\mathrm{PR}} (f) - (\expfct{i \psinoisy}, a^\sigma) \big\|_{\mathcal{Y},\beta}^2 = \|\mathrm{Sign} ( F(f) ) - \expfct{i \psinoisy} \|_Y^2  + \beta \, \| \, |f| - a^\sigma \|_{X}^2. \]
  Even though the latter functional by itself resembles a Tikhonov functional, we emphasize that here the factor $\beta>0$ acts as a balancing weight between two discrepancy terms and does not assume the role of a regularization parameter. This difference is reflected in our notation, where in the following we clearly distinguish between regularization parameters (denoted by $\alpha$) and the discrepancy weight $\beta$. Nevertheless, it proved reasonable to consider well-established regularization parameter choice rules also for choosing $\beta$, and we present a suitable example in Section \ref{sec:numerics}.

  The operator $\mathcal{F}_{\mathrm{PR}} : X \to \mathcal{Y}$ is, however, \emph{not} weakly sequentially closed as neither $\mathrm{Sign} (g)$ nor $|f|$ have this property, which is easily seen by considering the sequences $g_n \equiv \frac 1n$ and $f_n(t) = \expfct{int}$, respectively. Thus, following the approach suggested in \cite{FruSchLei05}, we approximate the Sign operator for $\varepsilon \to +0$ by
      \[ [\mathrm{Sign}_\varepsilon (g)] (\sY)  = \frac {g(\sY)}{\max(\varepsilon, |g(\sY)|)}. \]
  While the latter operators are Lipschitz-continuous with constant $\frac 1\varepsilon$, they still do not satisfy the sufficient condition for well-posedness of Tikhonov regularization.

  \begin{lemma}
    For any $\varepsilon >0$, the operator $\mathrm{Sign}_\varepsilon (g) : L^2_{\IC}(0,2) \to L^2_{\IC}(0,2)$ is not weakly sequentially closed.
  \end{lemma}

  \begin{proof}
    To prove the assertion we construct a counterexample. Let $K > \varepsilon$ be fixed and define the sequence $\{ g_n \}$ recursively by
	\[ g_0 (t) = \left\{ \begin{array}{ll} 	2K	& \mbox{ if } 0 < t < 2/3\\[2mm]
						-K	& \mbox{ if } 2/3 \leq t < 2 \end{array}\right. \quad \mbox{and} \quad
	   g_{n+1} (t) = \left\{ \begin{array}{ll} 	g_n(2t)		& \mbox{ if } 0 < t < 1\\[2mm]
							g_n(2t-1)	& \mbox{ if } 1 \leq t < 2. \end{array}\right.\]
    In fact, for these choices $g_n \wto 0$ holds, yet $\mathrm{Sign}_\varepsilon (g_n) \wto -\frac 13 \neq \mathrm{Sign}_\varepsilon (0)$, which shows that $\mathrm{Sign}_\varepsilon$ is not weakly sequentially closed.
  \end{proof}

  In a sense, the weak topology in $L^2_{\IC}(0,2)$ is insufficient to derive continuity results even for $\mathrm{Sign}_\varepsilon (g)$. A common remedy for obtaining regularization properties for operators that are continuous is to restrict ourselves to some subspace with compact embedding into $L^2_{\IC}(0,2)$. For the phase retrieval problem, we may thus formulate the following well-posedness result.

  \begin{proposition} \label{pro:reg_pr}
    Suppose that the kernel $k$ satisfies (\ref{eq:kernel}) and that the penalty term $\Reg (f)$ is stabilizing in $X_0 = H^1_{\IC}(0,1)$.
    Then for any $\varepsilon > 0$, $a^\sigma \in L^2 (0,1)$ and \mbox{$\psinoisy \in L^2(0,2)$}, minimizing the Tikhonov functional
      \begin{equation} \label{eq:TikhPR}
        \mathcal{T}_{\alpha, \beta}^{\sigma, \delta} (f) = \frac 12 \Big\|\mathrm{Sign}_\varepsilon ( F(f) ) - \expfct{i \psinoisy} \Big\|_Y^2  + \frac{\beta}2 \, \big\| \, |f| - a^\sigma \big\|_{X}^2 + \alpha \Reg (f)
      \end{equation}
    for $f \in X_0$ is well-posed in the sense of Remark \ref{rem:well-posed}.
  \end{proposition}

  \begin{proof}
    Due to the compact embedding of $X_0 \hookrightarrow X$, both operators
	\[ \mathrm{Sign}_\varepsilon ( F(f) ) : X_0 \to Y \qquad \mbox{and} \qquad |f| : X_0 \to X \]
    are strongly sequentially continuous. Hence they are in particular weakly sequentially closed.
    The result then follows from the identical argument as in Remark \ref{rem:well-posed}.
  \end{proof}

  \begin{remark}
    As our notation $\psinoisy$ suggests, the measurement errors affect the phase function data additively, i.e.,
      \[ \psinoisy =  \psi^\dag + \vartheta, \]
    where $\vartheta \in L^2(0,2)$ denotes the noise. Due to discontinuities in the principal part of the complex argument, however, we evaluate data discrepancy on the complex unit sphere $S_\IC^1$, where we have a multiplicative noise model,
      \[ \expfct{i \psinoisy} = \expfct{i \psi^\dag} \cdot \expfct{i \vartheta}. \]
    Finding a discrepancy term that suitably addresses this particular situation is certainly future work, but motivated by the regularization results in Proposition~\ref{pro:reg_pr} and the successful numerical experiments in Section~\ref{sec:numerics}, we use the subspace topology from $Y = L_\IC^2 (0,2)$ on $\{ \expfct{i \psi} \st \psi \in L^2(0,2) \} \subset Y$
    instead, and hence introduce the pseudo-metric $d(\psi_1, \psi_2) := \| \expfct{i \psi_1} - \expfct{i \psi_2} \|_Y$ for $\psi_1, \psi_2 \in L^2(0,2)$. Note that this distance penalizes phase differences modulo $2 \pi$ and does not increase the noise level, as
      \[ d(\psinoisy, \psi^\dag) = \Big\| 2 \sin \Big( \frac{\psinoisy - \psi^\dag}{2} \Big) \Big\|_{L^2(0,2)} \leq \big\| \psinoisy - \psidag \big\|_{L^2(0,2)}. \]
  \end{remark}

   As mentioned in Section \ref{s2}, the laser pulses generated in optical experiments exhibit a limited bandwidth; cf.~(\ref{supp1}). In the context of our current notation this requires that the searched-for function $f^\dag$ (which corresponds to the Fourier transform $\hat E$ of the laser pulse) can be neglected outside a certain parameter range $[\tau_1, \tau_2] \subset (0,1)$. Due to the structure of the autoconvolution operator, the same is to be expected for the image $y^\dag = F(f^\dag)$. Therefore, the phase data $\psinoisy \approx \arg (\ydag)$ carries little or no useful information in regions where $|y^\dag|$ is close to zero.
   As reliable measurements for $|y^\dag|$ are, however, not at hand, this fact is not accounted for in the Tikhonov-functional (\ref{eq:TikhPR}).
   To overcome this issue, a different data fidelity term has been proposed in \cite{BSKGBH15}, where the phase data discrepancy is weighted by $|F(f)|$ or, more precisely, by the normalized function $|F(f)|/\|F(f)\|_Y$. While this normalization is necessary to avoid an artificial bias towards reconstructions for which $|F(f)|$ is small, it also introduces a singularity for $\|F(f)\|_Y = 0$. Therefore, we again introduce an approximation level $\varepsilon > 0$ and consider the variational functional
      \begin{equation} \label{eq:TikhPRrel}
        \mathcal{T}_{\varepsilon, \alpha, \beta}^{\sigma, \delta} (f)
        = \left\{ \begin{array}{ll} \frac {\left \| F(f)  - |F(f)| \expfct{i \psinoisy} \right \|_Y^2}{\|F(f)\|_Y^2}  + \frac{\beta}2 \, \big\| \, |f| - a^\sigma \big\|_X^2 + \alpha \Reg (f), & \mbox{ if } \|F(f)\|_Y \geq \varepsilon,\\[2mm]
					  +\infty,   & \mbox{ else,} \end{array}\right.
      \end{equation}
   with $Y = L_\IC^2 (0,2)$. The proof that this functional also admits a minimizer in $H^1_{\IC}(0,1)$ is included in the Appendix.

   \begin{proposition} \label{pro:tikhrelmin}
    Suppose that the kernel $k$ satisfies (\ref{eq:kernel}) and that the penalty term $\Reg (f)$ is stabilizing in $X_0 = H^1_{\IC}(0,1)$ in the sense of Remark \ref{rem:well-posed}.
    Then for any $a^\sigma \in L^2 (0,1)$, \mbox{$\psinoisy \in L^2(0,2)$}, and $\varepsilon > 0$ such that
      \[ \mathrm{dom} \, \Reg(f) \cap \{ f \in X_0 \st \|F(f)\|_Y\geq \varepsilon \} \neq \emptyset, \]
    a minimizer of $\mathcal{T}_{\varepsilon, \alpha, \beta}^{\sigma, \delta} (f)$ defined by (\ref{eq:TikhPRrel}) exists in $X_0$.
   \end{proposition}

  \begin{remark}
    Another phenomenon of ill-posedness in solving the autoconvolution equation (cf.~\cite{FleiHof96,GHBKS13} and \cite[Example~3.2]{BuerHof15}) is due to amplitudes $a(\tX)$ in $f(\tX)=a(\tX)\,\expfct{i\,\varphi(\tX)}$. It occurs when $a(\tX)$ blows up to infinity locally near some $\tX_0 \in [0,1]$ in a way that hardly leaves effects on $F(f)$. This phenomenon, however, is suppressed in case of the phase retrieval problem because the amplitude function $a^\dagger$ is known up to measurement errors.
    It was, indeed, proven in \cite[Proposition~3]{BuerHof15} that even locally well-posed situations arise under the assumption that the amplitude function $a$ be fixed and essentially bounded.
  \end{remark}

\def\FwdD{F_n}

\section{Non-uniform rational B-splines (NURBS)}  \label{sec:NURBS}

  The numerical solution of the deautoconvolution problem \ref{ssec:deautoconv} as well as of phase retrieval problems \ref{ssec:phaseretrieval} requires discretization of the complex-valued autoconvolution equation with kernel function (\ref{eq:F}).
  The most natural choice are discretizations using piecewise constant functions, either in terms of step functions $\{ \chi_{[i/n, (i+1)/n)} \}_{i=0, \dots, n-1}$ (see, e.g., \cite{FleiHof96, Janno00, BuerFle14, Bue14}) or by means of Haar wavelets (e.g., \cite{ar10, WARH15}).
  While step functions yield simple (and computationally efficient) formulae, essentially reducing the continuous autoconvolution to its discrete counterpart, Haar wavelets are particularly suitable for the reconstruction of functions in $L^2(0,1)$ as they yield an orthonormal basis both of the infinite dimensional Lebesgue space as well as of its truncated, finite-dimensional approximations.
  Motivated by the results in Section \ref{s4} concerning the existence of minimizers of the Tikhonov functional in the smoother space $H_\IC^1(0,1)$ and by the underlying physical problem of ultrashort laser pulse characterization (see Section \ref{s2}), we focus on a different representation in the sequel, which is taylored towards the reconstruction of smooth functions.
  Namely, we will model the curve $(\RT f, \IT f)$ in the complex plane (and hence $f$ itself) as a rational B-spline curve.

  The shape of a non-uniform rational B-spline (short: NURBS) is determined by a set of $n$ \emph{control points} $\{ P_1, \dots, P_n \} \subset \IR^2$, corresponding positive \emph{weights} $w = \{ w_1, \dots, w_n \}$ and a non-decreasing \emph{knot vector} $\knot = \{ \knot_1, \dots, \knot_{n + \p +1} \}$, where $\p$ is the polynomial degree of the spline.
  From the knot vector $\knot$, the B-spline basis functions $N_{j,\p}$ are determined via the \emph{Cox-de Boor} recursion
    \begin{align*}
      N_{j, 0}    (\tX) & = \Big\{ \begin{array}{ll} 1 & \mbox{ if } \knot_j \leq \tX < \knot_{j+1} \\ 0 & \mbox{ otherwise,}  \end{array}\\
      N_{j, \p} (\tX) & = \frac{\tX - \knot_{j}}{\knot_{j + \p} - \knot_{j}} N_{j,\p-1} (\tX) + \frac{\knot_{j + \p +1} - \tX}{\knot_{j + \p +1} - \knot_{j+1}} N_{j+1,\p-1} (\tX).
    \end{align*}
  Notice that throughout this section we adopt the convention $\frac 00 := 0$ as is customary in this context. The NURBS curve is then given by
    \[ \gamma [P, w] (\tX) = \frac{\sum_{j=1}^n P_j w_j N_{j,\p} (\tX)}{\sum_{j=1}^n w_j N_{j,\p} (\tX)}, \qquad \tX_{\min} \leq \tX \leq \tX_{\max}, \]
  where $\tX_{\min} := \knot_{\p+1}$, $\tX_{\max} := \knot_{n+1}$.  We refer the interested reader to \cite{pt97, r01} for a comprehensive introduction to NURBS.
   Defining \emph{rational basis functions} as
	\[ R_{j,\p} (\tX) = \frac{w_j N_{j,\p} (\tX)}{\sum_{l=1}^n w_l N_{l,\p} (\tX)}\]
  the curve can be equivalently written as
    \[ \gamma [P, w] (\tX) = \sum_{j=1}^n P_j R_{j, \p} (\tX). \]
  It can be easily seen from these definitions that NURBS curves are invariant (up to re-parametrization) under rescalings and shifts of the knot vector.
  Without loss of generality, we may thus assume that the knot vector satisfies $\knot_{\p+1} = 0$ and $\knot_{n+1} = 1$, so that the NURBS curve is parametrized by $\tX \in [0, 1]$.
  A common choice are \emph{open} knot vectors of the form
  	\[ \knot = (0, \dots, 0, \knot_{p+2}, \dots, \knot_n, 1, \dots, 1), \]
  which have the additional property that the resulting NURBS curve $\gamma(\tX)$ begins and ends in the first and last control point, respectively, i.e., it satisfies $\gamma(0) = P_1$ and $\gamma(1) = P_n$.

\subsection*{Autoconvolution and NURBS}

  For planar curves the NURBS control points $P_j$ are typically assumed to be in $\IR^2$. Nevertheless, the formulae and results in Section~\ref{sec:NURBS} remain valid, if we model them as $P_j \in \IC$ which allows to conveniently view the parametrized spline curves as complex-valued functions.
  In addition, we shall refer to the real and imaginary parts of the control points by $\Pre_j$ and $\Pim_j$, respectively, so that $P_j = \Pre_j + i \Pim_j$.
  The main motivation for working with real parameters is the fact that real-valued functionals, such as the variational Tikhonov-type objective functionals defined in Section \ref{s4}, are not holomorphic.
  When using gradient based optimization methods, it is therefore necessary to consider partial derivatives with respect to $\Pre$ and $\Pim$.

  Now let the spline degree $\p$, the number $n \geq p+1$ of control points, as well as the knot vector $\knot$ be fixed. Then, we denote the finite dimensional space of NURBS \emph{design parameters} $x = (\Pre, \Pim, w)$ by
     \[ X_n := \IR^n \times \IR^n \times \IR_+^n ~ \subset \IR^{3n} \]
  and by $\gamma : X_n \to X = L_\IC^2 (0, 1)$ the synthesis operator
    \[ \gamma [x] (\tau) = \sum_{j = 1}^n (\Pre_j + i \Pim_j) R_{j,p}(\tau), \qquad \mbox{with }  R_{j,\p} (\tX) = \frac{w_j N_{j,\p} (\tX)}{\sum_{l=1}^n w_l N_{l,\p} (\tX)}. \]
  In terms of the complex-valued NURBS curves, the autoconvolution operator maps as
	\begin{equation} \label{eq:fwdop}
		\FwdD = F \circ \gamma : X_n \to Y = L_\IC^2(0, 2),
	\end{equation}
  where $F : X \to Y$ is given by (\ref{eq:F}) with $f = \gamma [x]$. In this way, we aim to reconstruct approximations $f$ in the discretized space
      \[ X_{n, p}^\eta : = \left\{ \gamma [x]  ~ : ~  x \in X_n  \right\} \subset H_\IC^1(0,1). \]

\bigskip

\subsection*{Regularization for NURBS design parameters}

\def\cb{{(\wa)^{-4}}} 
\def\rwa{{-2}}
\def\rwaa{2}
\def\rwb{{2}}
\def\wa{{w_0/\rwaa}}

  As illustrated in Section~\ref{s4}, the regularity of the forward operator for the phase retrieval problem~\ref{ssec:phaseretrieval} in its continuous form guarantees existence of minimizers of Tikhonov-type functionals in spaces that embed compactly into $L_\IC^2 (0,1)$. In choosing a discretization space that consists of sufficiently regular spline curves, we readily ensure that the reconstructed solutions belong to $H_\IC^1 (0,1)$.
  As penalty term for the NURBS design parameters $x = (\Pre, \Pim, w)$, defined in the previous subsection, we propose
	\begin{equation} \label{eq:NURBSpenalty}
	  \Reg(x) = \beta_P \Reg_P (\Pre, \Pim) + \beta_w \Reg_{w_0} (w),
	\end{equation}
  where we penalize the distance between control points $P_j = \Pre_j + i \Pim_j$, and hence to some extent the length of the curve (corresponding to the $H_\IC^1 (0,1)$ seminorm), by
	\[ \Reg_P (\Pre, \Pim) = \frac 1{2 n} \left[ \sum_{j=1}^{n-1} (\Pre_{j+1} - \Pre_j)^2 + (\Pim_{j+1} - \Pim_j)^2 \right]. \]
  To account for the constraint that the NURBS weights are required to be positive we penalize them by
  \begin{align*}
    \Reg_{w_0} (w) & =  \frac 1{\rwaa n} \sum_{j=1}^n f_{w_0}(w_j),
  \end{align*}
  where
    \[ f_{w_0} (w) = \left\{ \begin{array}{ll} 	+\infty				& \mbox{ if } w \leq 0 \\
						w^\rwa				& \mbox{ if } 0 < w < w_0/2 \\
						\cb (w - w_0)^\rwb \qquad 	& \mbox{ if } w \geq w_0/2. \end{array}\right.\]
  Clearly, this functional, which is shown in Figure \ref{fig:fw}, acts as a (quadratic) barrier for \mbox{$w \to +0$} as well as for \mbox{$w \to +\infty$}. In addition, it gives preference to values near a reference weight $w_0 > 0$.
  It is worth noting, however, that $w_0$ does not represent {\em a priori} knowledge on the average size of the weights, because NURBS weights are only determined up to a constant factor. Instead, the choice of $w_0$ in combination with the size of the parameter $\beta_w$ typically impacts the relative difference $\max(w)/\min(w)$ in the approximate solutions.

  \begin{figure}[ht]
	\includegraphics{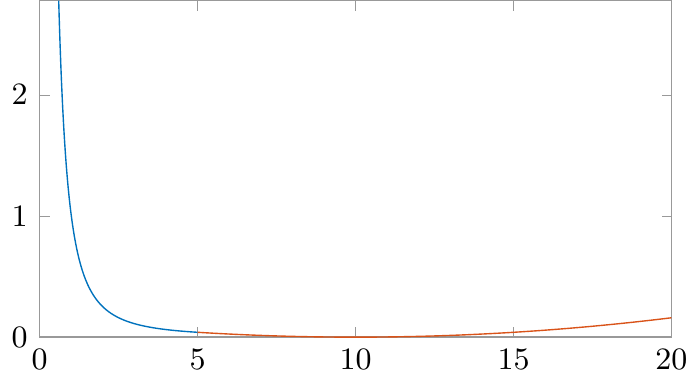} ~ \includegraphics{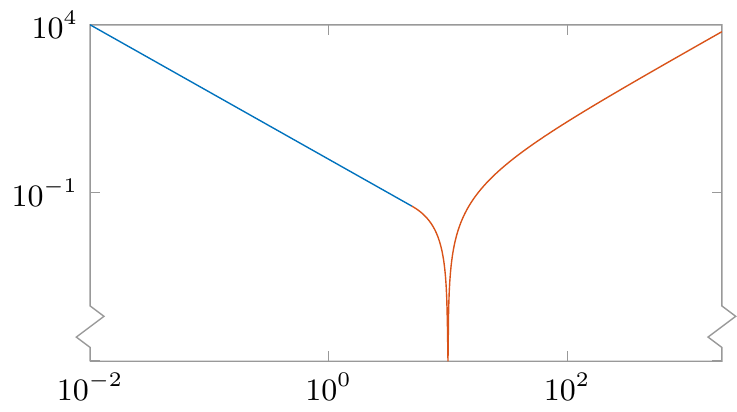}
    \caption{Regularizing functional $f_{w_0} (w)$ for $w_0 = 10$ (left) and its log-log plot (right).}
     \label{fig:fw}
  \end{figure}

\section{Numerical results}   \label{sec:numerics}

\newcommand{\regsol}[1]{#1_{\alpha, \beta}^\noiselevel}   
\newcommand{\regsols}[1]{#1_{\alpha, \beta^*}^\noiselevel}  
\newcommand{\regsolit}[2]{#1_{\alpha, \beta_{#2}}^\noiselevel}

\def\fad{\regsol{f}}
\def\fads{\regsols{f}}
\def\noiselevel{{\sigma, \delta}}
\def\Tad{\regsol{\mathcal{T}}}
\def\xad{\regsol{x}} 
\def\xads{\regsols{x}} 

\def\imgscw{0.49\textwidth}
\def\imgsch{100pt}

  We have tested our method both on synthetic data as well as on real data obtained from an SD-SPIDER apparatus at the Max-Born-Institute for Nonlinear Optics and Short Pulse Spectroscopy in Berlin, Germany. The regularized solutions are NURBS design matrices $\xad = (\regsol{u}, \regsol{v}, \regsol{w})$, minimizing a variational Tikhonov-type functional
      \[ \Tad (x) = \frac{\left\| \FwdD(x) - |\FwdD(x)| \expfct{ i \psinoisy} \right\|_Y^2}{\|\FwdD(x)\|_Y^2} + \frac{\beta}2 \big\| \, |\gamma[x]| - \absnoisy \big\|_X^2 + \alpha \Reg (x), \]
  with penalty term given by (\ref{eq:NURBSpenalty}).
  The corresponding functions in $X = L_\IC^2 (0,1)$ are obtained as $\fad = \gamma [\xad]$. To be precise, we used the functional (\ref{eq:TikhPRrel}) with approximation level $\varepsilon = 10^{-10}$, but as situations with $\| \FwdD(x) \|_Y < \varepsilon$ never occured during the computations, we shall omit this additional parameter for the sake of brevity.

  For finding an approximation of $\xad$ we use a Quasi-Newton method,
    \begin{equation} \label{eq:qNwt}
      x_{j+1} = x_j + \rho_j H_j^{-1} \grad \Tad (x_j),
    \end{equation}
  with step-sizes $\rho_k$ satisfying the Wolfe conditions and BFGS-updates of the approximation $H_k$ of the Hessian.
  It is well-known, that gradient based optimization methods for nonlinear problems generally suffer from local minima \cite{BeKli12}. To obtain a better approximation of the global optimizer, we employ a strategy known as TIGRA (the name being derived from TIkhonov-GRAdient method), which was introduced by Ramlau \cite{Ramlau02}. This method was proven to converge globally for suitable problems and with suitably chosen parameters in \cite{Ramlau02, Ramlau03, WARH15}.

   Aiming at a well-balanced data fit both for $\psinoisy \approx \arg (y^\dag)$ and for $\absnoisy \approx |f^\dag| = |\gamma[x^\dag]|$, we used a similar approach to TIGRA, but with respect to the weight $\beta$ balancing the discrepancy terms (cmp.~Section~\ref{s4}). Starting with large $\beta_0$ and $q<1$ (here: $\beta_0 = 100$, $q = 0.25$), the resulting method is described as follows.
    \begin{itemize}
     \item Find $\regsolit{x}{k} \in \argmin \regsolit{T}{k} (x)$ using (\ref{eq:qNwt}) starting from $x_{k,0} = \regsolit{x}{k-1}$,
    \end{itemize}
   where $\beta_{k} = q \, \beta_{k-1}$.
   Initially, for large $\beta_0$ only $\absnoisy$ is emphasized while reconstructions will not typically provide a good match for $\psinoisy$. Then the weight is gradually shifted giving more and more importance to $\psinoisy$ while the initial good match for $\absnoisy$ declines only mildly. A globalization approach (such as the proposed reweighing of the discrepancy terms, for example) is certainly required to reach an approximate solution of the original problem, but comes at the cost of a higher computational effort. In our case the additional effort is due to the repeated optimization with different values of the weight $\beta_k$.
   To keep the number of iterations to a minimum, however, we may solve the earlier optimization problems inexactly and increase the required precision while the iteration proceeds. In our experiments we have employed this technique using \verb-tol-$_k = \max (\frac {\beta_k}{2000}, 10^{-9})$ as stopping tolerance for $\grad \regsolit{T}{k} (x_{k,j})$, limiting the number of iterations during each optimization procedure to \verb-maxIt- $= 10000$.

   The opposing trends of the two data fidelity terms (cmp.\,Figure \ref{fig:pcr}) during the iteration can be exploited to obtain a stopping rule for the procedure. Considering that our main objective is to optimize the \emph{overall} data fit, we define a weighted relative least-squares functional
    \begin{equation} \label{eq:tls}
      e(x)^2 := 2 \, d(x)^2 + r(x)^2,
    \end{equation}
   in terms of the auxilliary quantities
    \[ d(x)^2 := \frac{\left\| \FwdD (x) - |\FwdD(x)| \expfct{ i \psinoisy} \right\|_Y^2}{\|\FwdD(x)\|_Y^2} \qquad \mbox{and} \qquad r(x)^2 :=  \frac{\left\| \, |\gamma[x]| - \absnoisy \right\|_X^2}{\| \absnoisy \|_X^2}, \]
  and stop the iteration when $e\big(x_{\alpha, \beta_k}^\noiselevel \big)$ reaches its lowest value. Putting more emphasis on the discrepancy term $d(x)$ in $Y$ proved beneficial in the numerical experiments. We remark that this procedure may be regarded as a bilevel optimization approach for choosing the discrepancy weight. Namely, for fixed $\alpha > 0$, we would choose $\beta^*$ as an approximate solution of
  	\begin{align*}
  	  & e(\xad)^2   \to \min_{\beta >0}\\
  	  \mbox{such that } & \xad \in \argmin_{x \in X_n} \Tad (x).
  	\end{align*}
  For our purposes, however, the primary motivation is to improve the global convergence properties of the Quasi-Newton method rather than to solve the latter problem. The final values from the experiments in the following subsections have been collected in Table \ref{tab:pcr}.

  \begin{figure}
	\includegraphics{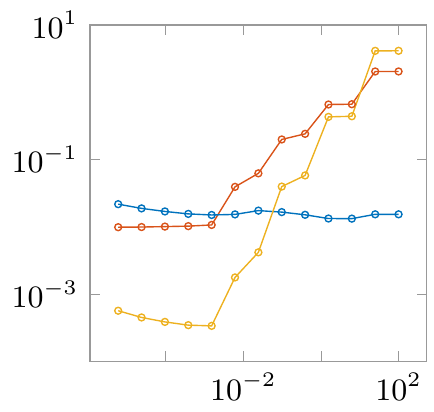} ~ \includegraphics{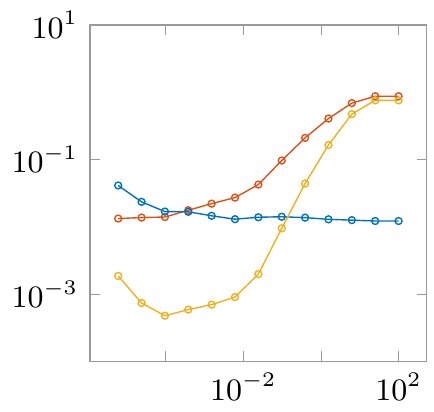} ~ \includegraphics{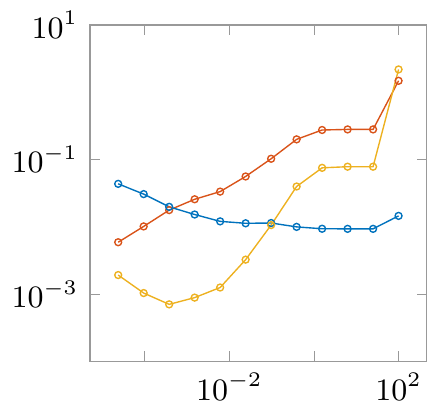}
    \caption{Relative data error quantities $d(\xad)$ (red) and $r(\xad)$ (blue) from the following subsections in dependence on $\beta$. The final value $\beta^*$ is chosen as minimizer of \mbox{$e(\xad)^2$} (yellow).}
     \label{fig:pcr}
  \end{figure}

   We first present some results obtained with synthetic data. In order to stay as close to the realistic data situations as possible, we use the same kernel function throughout the following subsections.
   The kernel shown in Figure \ref{fig:kernel} was obtained from physical modelling of the nonlinear optical processes that result in the measurements (cf.\,\cite{KokBirBetGreSte10, GHBKS13, BSKGBH15}).
   The target phase function $\varphi^\dag = \arg (f^\dag)$ was chosen identical to \cite[Section~3.2]{Ge11} (see also \cite{GHBKS13}), and the noisy data were generated by adding 1\% relative noise.

  \begin{figure}
    \begin{center}
      \includegraphics[width=\imgscw, height=\imgsch]{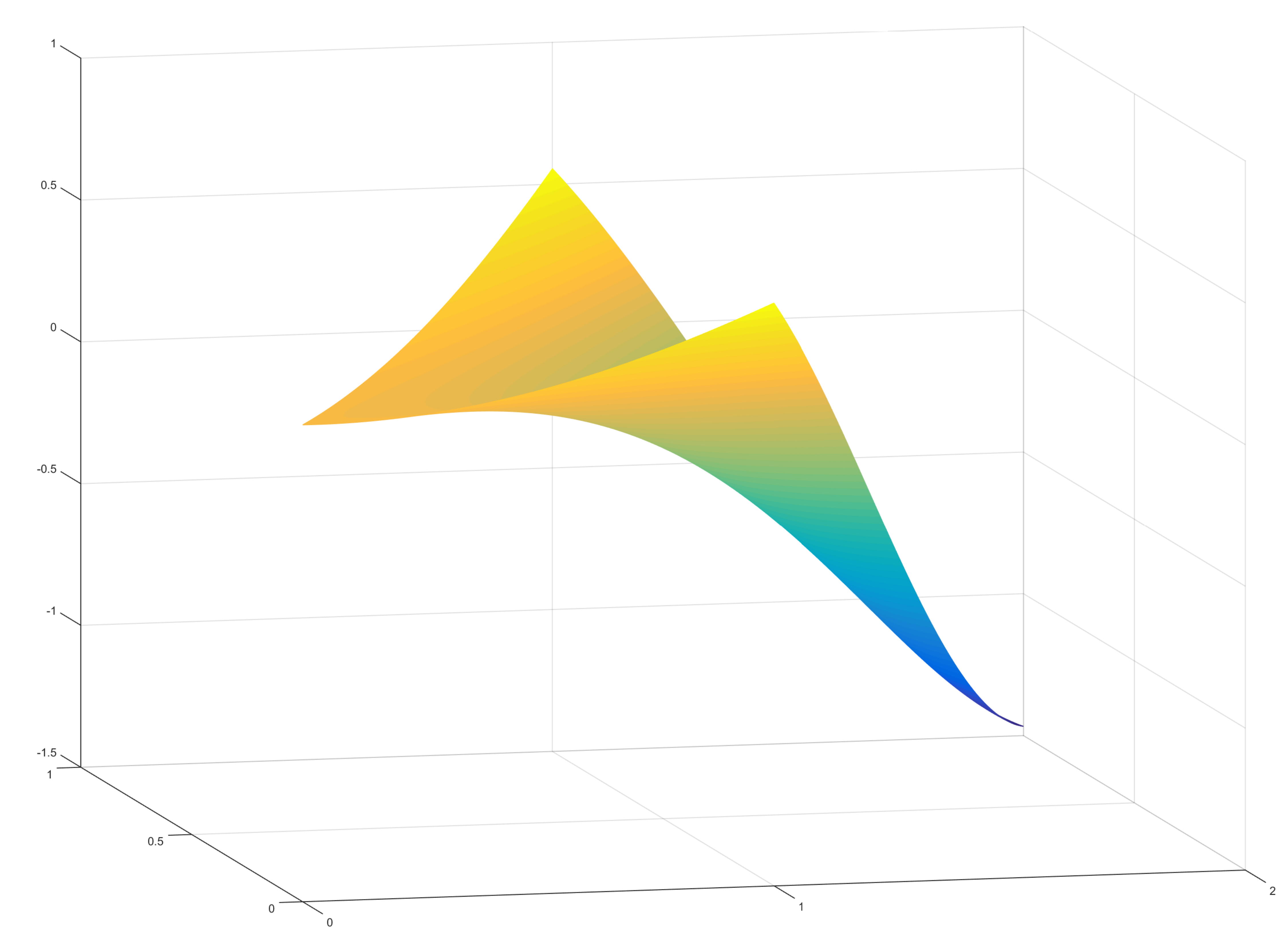}  ~  \includegraphics[width=\imgscw, height=\imgsch]{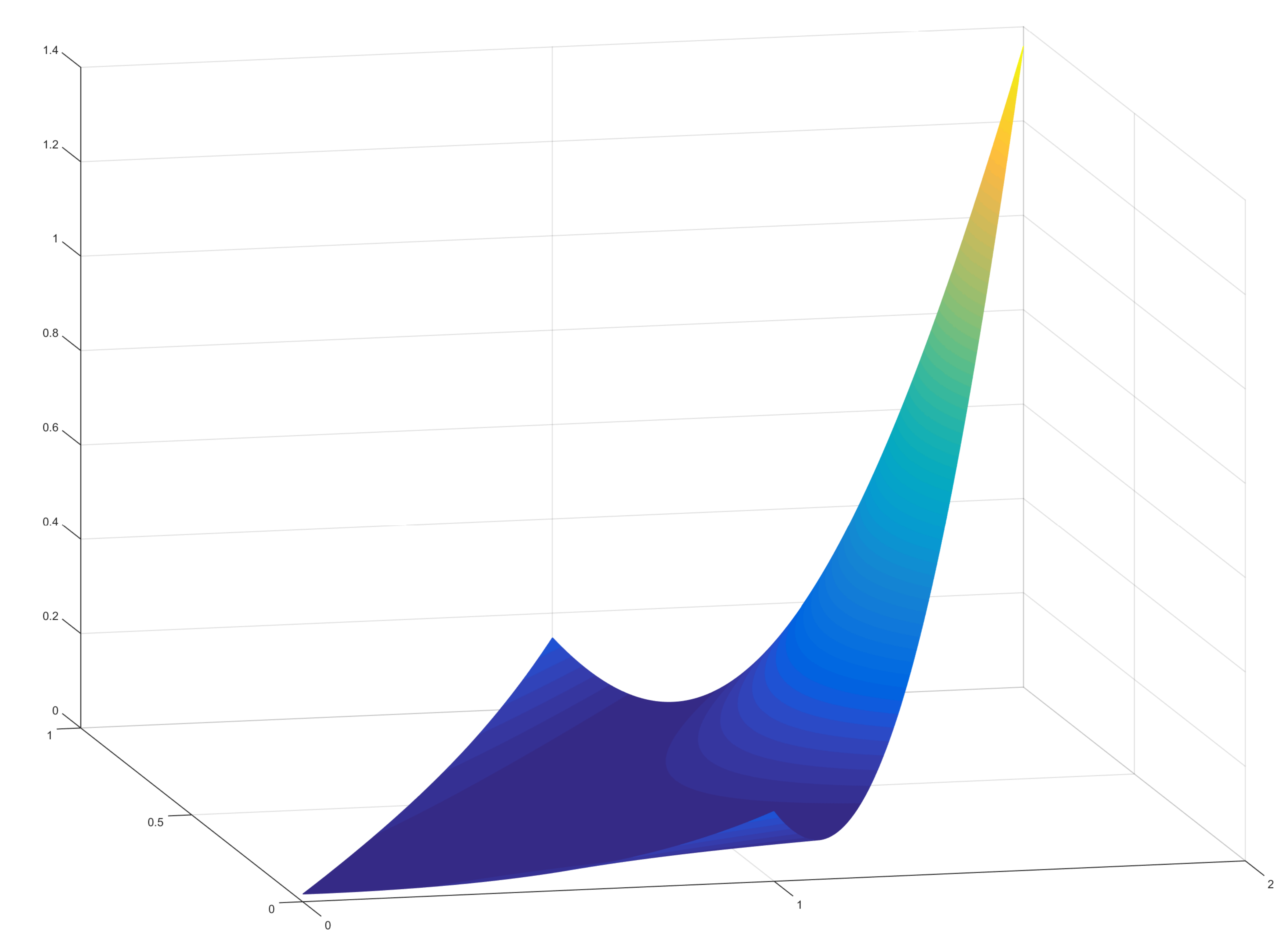}
    \caption{The real (left) and imaginary (right) part of the kernel function.}
     \label{fig:kernel}
    \end{center}
  \end{figure}

  The approximate solutions were represented as complex-valued NURBS curves $f = \gamma [x]$ as introduced in Section \ref{sec:NURBS}. In our experiments we used quadratic splines, i.e., $p=2$, with $n=150$ control points and corresponding positive weights. The knot vector $\knot$, which determines where and how the control points impact the parametrized curve, was chosen as \emph{open uniform}. This is to say that
      \[ \knot_1 = \dots = \knot_{p+1} = 0, \qquad \knot_{n+1} = \dots = \knot_{n+p+1} = 1, \]
  and
      \[ \knot_j = \frac{j-p-1}{n-p} \qquad \mbox{for } p+2 \leq j \leq n. \]
  In $X= L_\IC^2 (0,1)$ we have discretized the resulting curves by choosing $N = 1000$ equidistant sampling points $\tau_k = \frac{k-1}{N-1}$ and, correspondingly, in the image space $Y = L_\IC^2 (0,2)$ by $s_m = \frac{m-1}{N-1}$ for $m = 1, \dots, 2N-1$. Computations were started with an initial guess consisting of constant weights $w_j = w_0 = 10$ and control points $P_j = u_j + i v_j$ interpolated from $\absnoisy$ such that $u_j = v_j \approx \sqrt 2 \cdot \absnoisy \big(\frac{j-1}{n-1}\big)$. In the penalty term (\ref{eq:NURBSpenalty}) we also used $w_0 = 10$ and balanced both terms equally by $\beta_P = \beta_w = 1$. The regularization parameter was chosen as $\alpha = 10^{-6}$.

  \begin{table}[hbpt]
    \begin {center}
    \begin{tabular}{*{7}{r}} 
      \hline \noalign{\smallskip}
			    Data &  CPU-time &   \#Iter & 	      $\beta^*$ &       $d(\xads)$ &         $r(\xads)$ &     $e(\xads)^2$\\
  \hline \noalign{\smallskip}
       $(\absnoisy, \yd)$	 &    1585 s &    19942 &  $3.81 \cdot 10^{-4}$ &        $1.02 \cdot 10^{-2}$ &          $1.55 \cdot 10^{-2}$ &      $4.47 \cdot 10^{-4}$\\
       $(\absnoisy, \psinoisy)$	 &    1663 s &    19391 &  $9.54 \cdot 10^{-5}$ &        $1.39 \cdot 10^{-2}$ &          $1.68 \cdot 10^{-2}$ &      $6.69 \cdot 10^{-4}$\\
       measured			 &    1848 s &    24403 &  $3.81 \cdot 10^{-4}$ &        $1.77 \cdot 10^{-2}$ &          $1.97 \cdot 10^{-2}$ &      $1.01 \cdot 10^{-3}$\\
  \hline
    \end{tabular}\\[5mm]
    \caption{Comparison of results for various data situations from the following subsections with $\beta_0 = 100$, $\alpha = 10^{-6}$.} \label{tab:pcr}
    \end{center}
  \end{table}

  \bigskip

\subsection*{Phase retrieval with synthetic data for y}

\def\imgscw{0.49\textwidth}
\def\imgsch{100pt}
\def\data{absX_fullY}

  We first consider the ideal data situation, assuming that measurements are available for both the modulus and the argument of $\yd \approx y$ as well as for the modulus $\absnoisy \approx |f^\dag| = |\gamma [x^\dag]|$.
  To simulate the real data situation we used the same measurements for $\absnoisy$ as in the final subsection. These measurements are shown in Figure \ref{fig:X_synth_full} together with the reconstructed solution.
  \begin{figure}
   \begin {center}
    \begin{tabular}{cc}
     \includegraphics{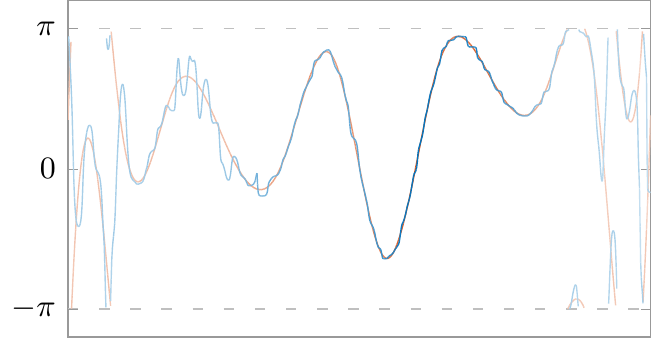} & \includegraphics{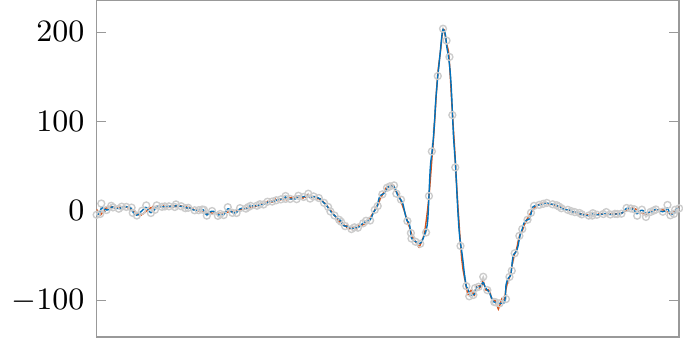}\\
     \includegraphics{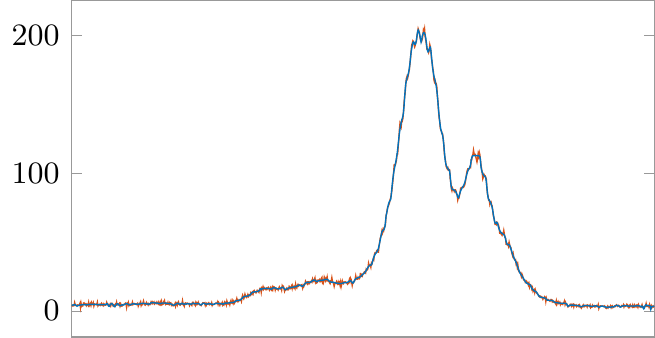} &  \includegraphics{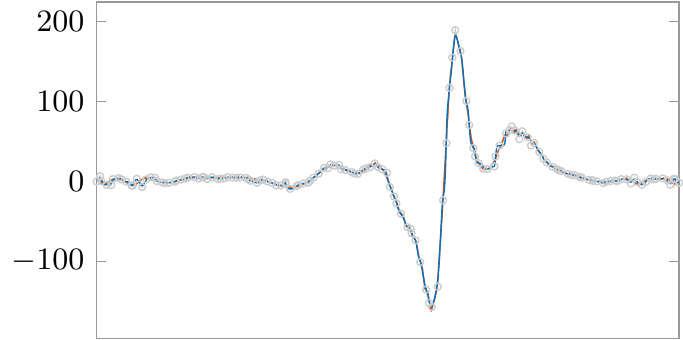}
    \end{tabular}
    \caption{Top-left: Target phase $\varphi^\dag$ (red) and reconstruction $\arg (\fads)$ (blue). Bottom-left: Data $\absnoisy \approx |f^\dag|$ (red) and reconstruction $|\fads|$  (blue). Right column: Real (top) and imaginary (bottom) part of $f^\dag$ (red) and $\fads$ (blue) together with NURBS control points.}
     \label{fig:X_synth_full}
   \end{center}
  \end{figure}

  \begin{figure}
    \begin{center}
     \includegraphics{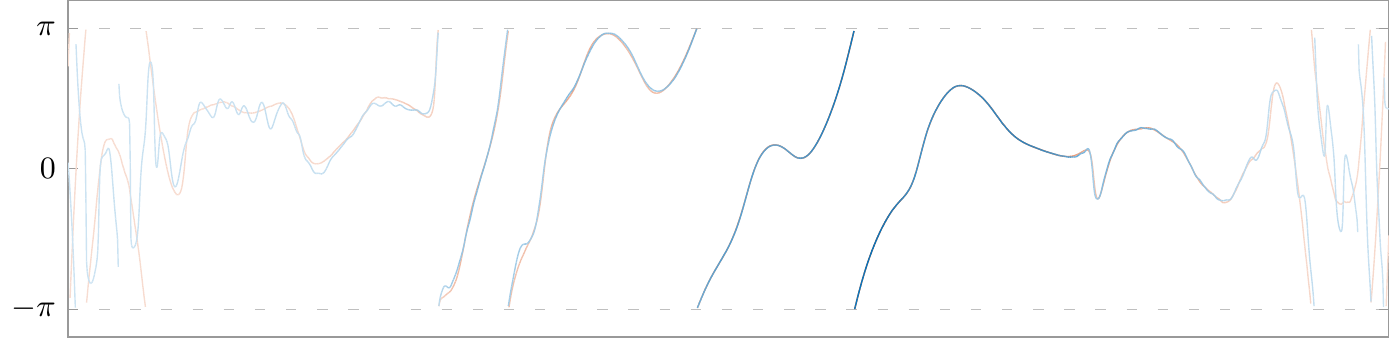}\\
     \includegraphics{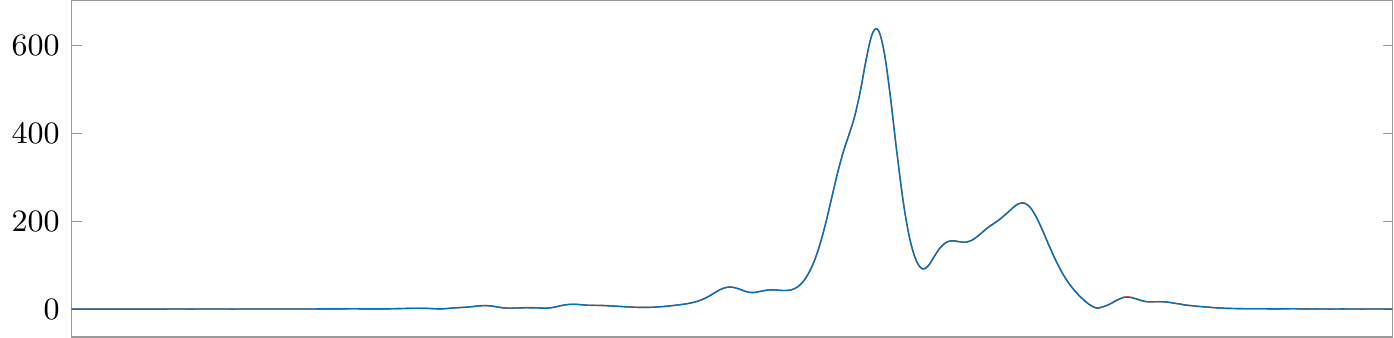}
    \caption{Top: True $\arg (y^\dag)$ (red) and reconstructed $\arg \FwdD(\xads)$ (blue). 
	      The brightness is chosen in (logarithmic) dependence on $|\FwdD(\xads)|$. Bottom: True $|y^\dag|$ (red) and $|\FwdD(\xads)|$ (blue)}
     \label{fig:Y_synth_full}
    \end{center}
  \end{figure}

  In order to be able to work with the same parameter values for $\beta_0$, \verb-tol-$_k$, and $\alpha$ as in the other test cases, we consider here the Tikhonov functional given by
      \[ \Tad (x) = \frac{\left\| \FwdD(x) - \yd \right\|_Y^2}{\|\yd\|_Y^2} + \frac{\beta}2 \big\| \, |\gamma[x]| - \absnoisy \big\|_X^2 + \alpha \Reg (x), \]
  where the NURBS penalty term is as defined by (\ref{eq:NURBSpenalty}) and $\alpha = 10^{-6}$. Similarly, the relative data misfit term
      \[ d(x) = \frac{\left\| \FwdD(x) - \yd \right\|_Y}{\|\yd\|_Y} \]
  was used in the least-squares error functional $e(x)$ defined by (\ref{eq:tls}).

  The results are shown in Figures \ref{fig:X_synth_full} and \ref{fig:Y_synth_full}.
  Note that the real and imaginary parts of the reconstruction $\fads = \gamma[\xads]$ (right column of Figure~\ref{fig:X_synth_full}) evidently provide a good approximation of $f^\dag$. Small oscillations in regions where the modulus $|f^\dag|$ is close to zero, however, result in quite large deviations of their arguments in these areas (left column of Figure~\ref{fig:X_synth_full}). As one might expect reconstructions can therefore only be reliable in those regions where $|\fads| \gg 0$. To emphasize this observation we have weighted the brightness of the plot of $\arg \big( \fads \big)$ and $\arg (f^\dag)$ depending on their absolute values.

  \bigskip

\subsection*{Phase retrieval with synthetic data for arg(y)}

\def\imgscw{0.49\textwidth}
\def\imgsch{100pt}
\def\data{absX_phaseY}

  \begin{figure}
   \begin {center}
    \begin{tabular}{cc}
     \includegraphics{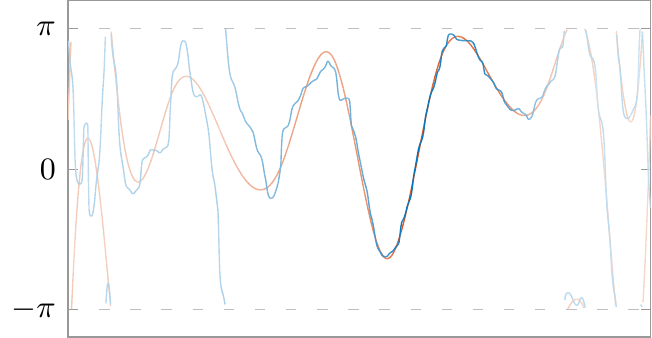} & \includegraphics{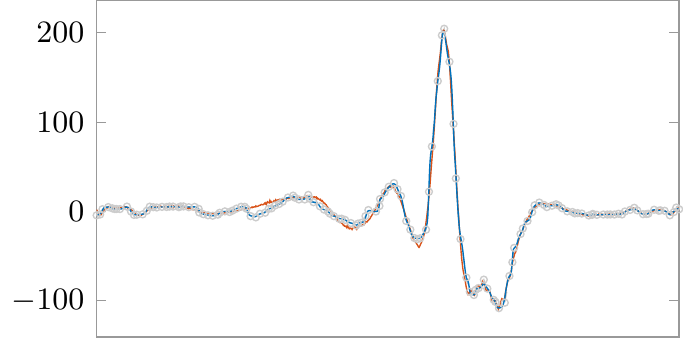}\\
     \includegraphics{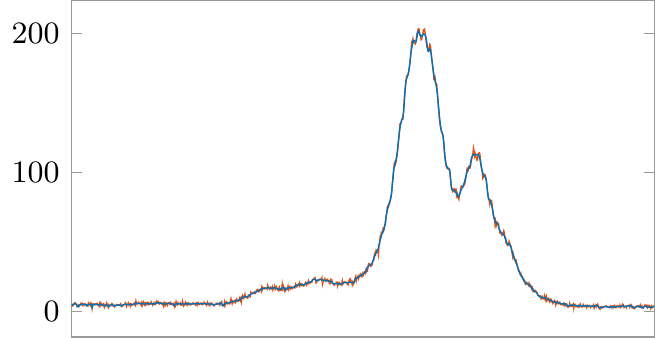} & \includegraphics{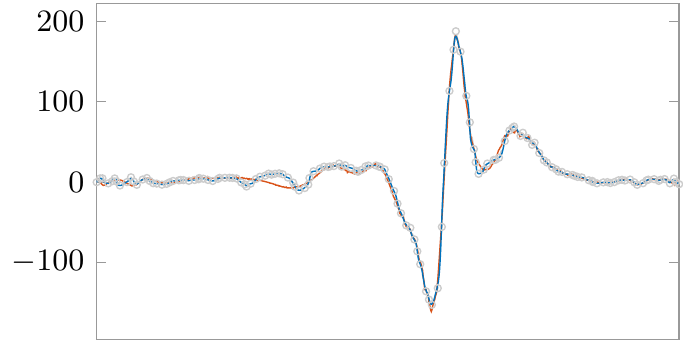}
    \end{tabular}
    \caption{Top-left: Target phase $\varphi^\dag$ (red) and reconstruction $\arg (\fads)$ (blue). Bottom-left: Data $\absnoisy \approx |f^\dag|$ (red) and reconstruction $|\fads|$  (blue). Right column: Real (top) and imaginary (bottom) part of $\fads$ together with NURBS control points.}
     \label{fig:X_synth} \label{fig:Xcomplex_synth}
   \end{center}
  \end{figure}

  As another academic example, we consider the phase retrieval problem corresponding to the real data situation in the following subsection. We assume that measurements are available for the argument of $\yd \approx y$ and for the modulus $\absnoisy \approx |f^\dag| = |\gamma [x^\dag]|$.
  Again, we used the available experimental data for $\absnoisy$ which is shown in Figure~\ref{fig:X_synth}.

  The results are shown in Figures \ref{fig:X_synth} and \ref{fig:Yphase_synth}.   When comparing to the ideal data situation where both $|y^\dag|$ and $|f^\dag|$ are approximately known, a certain fall-off in quality is evident. However, especially in those regions where $|f^\dag| \gg 0$, the reconstruction still provides a good approximation of the target.

  \begin{figure}
    \begin{center}
     \includegraphics{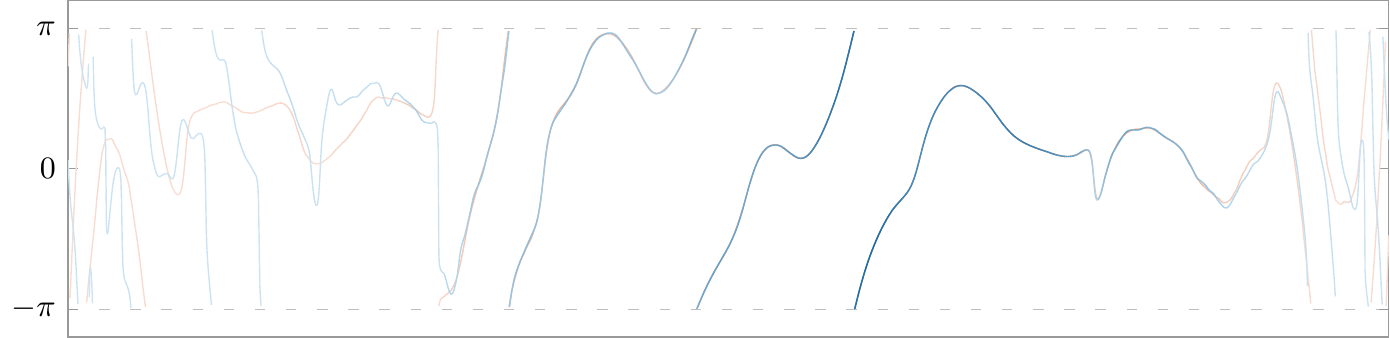}\\
     \includegraphics{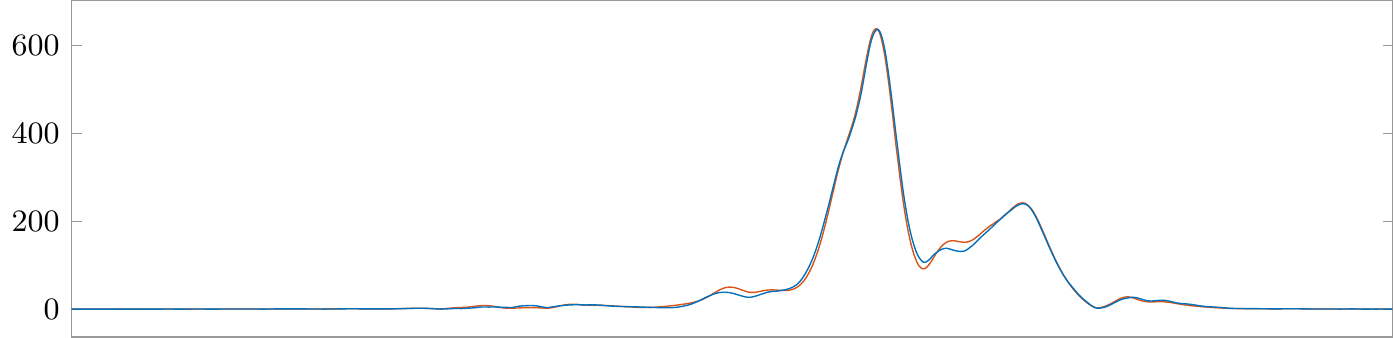}
    \caption{Top: Generated data $\psinoisy$ (red) and reconstruction (blue) for $\arg \FwdD(\xads)$.
	      The brightness is chosen in (logarithmic) dependence on $|\FwdD(\xads)|$. Bottom: Target $|y^\dag|$ and reconstructed $|\FwdD(\xads)|$.}
     \label{fig:Yphase_synth}
    \end{center}
  \end{figure}

  \bigskip

\subsection*{Phase retrieval with real data}

   In the real data from optical measurements, the frequency band containing the support of the $\absnoisy \approx |f^\dag|$ is located in between $\omlow = 3.5 \times 10^{15}$\,Hz and $\omup = 4.1 \times 10^{15}$\,Hz, and the frequency of the continuous wave (see\,Section~\ref{s2}) at $\omcw \approx 3.86 \times 10^{15}$\,Hz.

   As above, the Tikhonov functional is given by
      \[ \Tad (x) = \frac{\left\| \FwdD(x) - |\FwdD(x)| \expfct{ i \psinoisy} \right\|_Y^2}{\|\FwdD(x)\|_Y^2} +  \frac{\beta}2 \big\| \, |\gamma[x]| - \absnoisy \big\|_X^2 + \alpha \Reg (x), \]
   with penalty term $\Reg(x)$ defined by (\ref{eq:NURBSpenalty}) and $\alpha = 10^{-6}$. For the TIGRA-type approach with respect to the weight $\beta$, we used $\beta_0 = 10^2$ and $q = 0.25$.
   Keeping in mind that phase data is reliable only in regions where the absolute value of the function is sufficiently large, we again observe a good data fit in Figures \ref{fig:X} and \ref{fig:Yphase}.
   To highlight this observation, we compare several reconstructions, which were obtained using different parameters in Figure~\ref{fig:Xcomp}.

\def\data{meas}

  \begin{figure}[tbp]
   \begin {center}
    \begin{tabular}{cc}
     \includegraphics{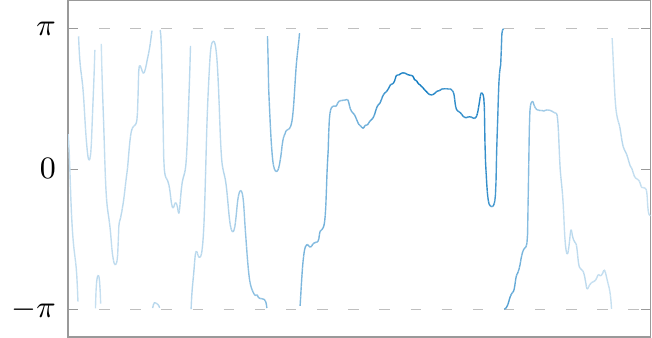} & \includegraphics{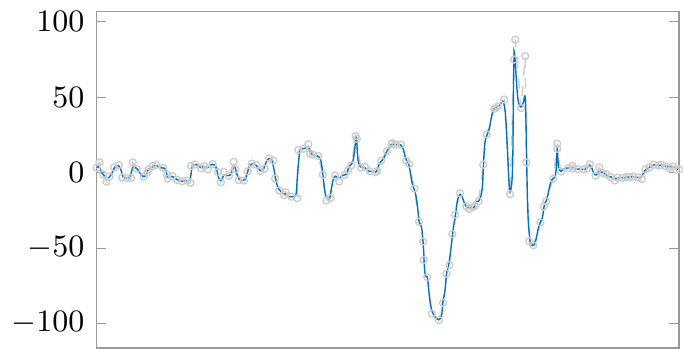}\\
     \includegraphics{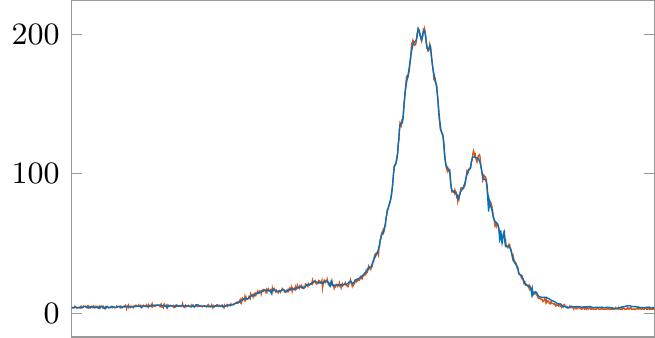} & \includegraphics{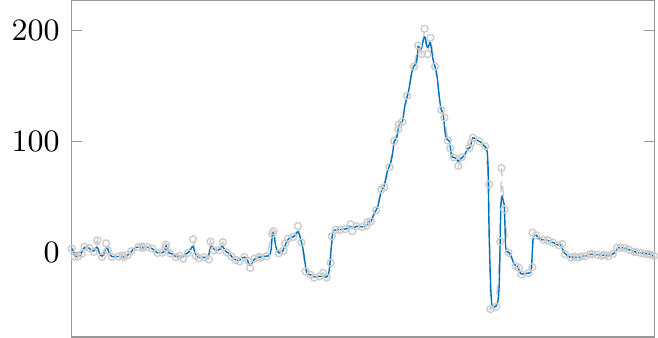}
    \end{tabular}
    \caption{Top-left: Reconstructed phase $\arg (\fads)$. Bottom-left: Data $\absnoisy \approx |f^\dag|$ (red) and reconstruction $|\fads|$  (blue). Right colomn: Real (top) and imaginary (bottom) part of $\fads$ together with NURBS control points.}
     \label{fig:X} \label{fig:Xcomplex}
   \end{center}
  \end{figure}

  \begin{figure}[tbp]
    \begin{center}
     \includegraphics{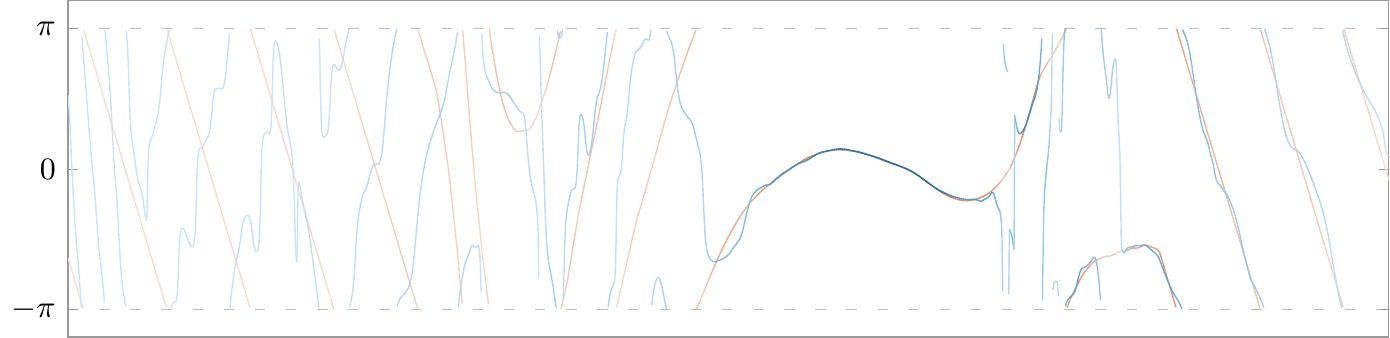}\\
     \includegraphics{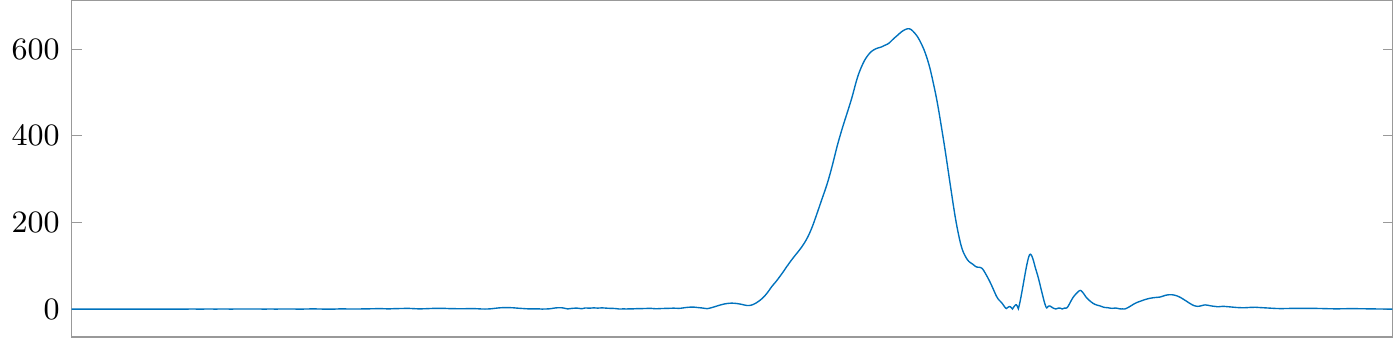}
    \caption{SD-phase data $\psi^\delta$ (red) and reconstruction (blue) for $\arg \FwdD(x_{\alpha^*}^\delta)$. The brightness is chosen in (logarithmic) dependence on $|\FwdD(x_{\alpha^*}^\delta)|$.}
     \label{fig:Yphase}
    \end{center}
  \end{figure}

  \begin{figure}[tbp]
    \begin{center}
     \includegraphics{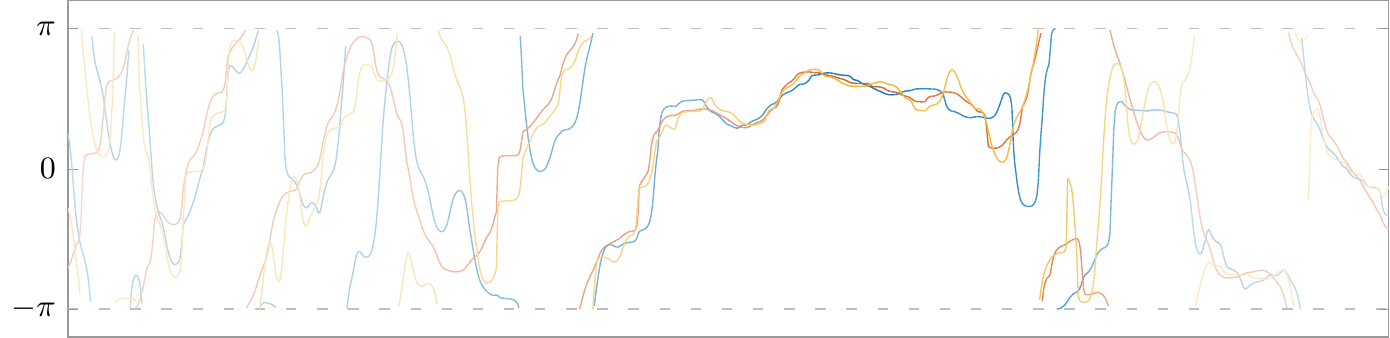}\\
    \caption{Comparison of various phase reconstructions $\arg (\fads)$ for $\alpha = 10^{-6}$ (blue), $5 \cdot 10^{-6}$ (yellow), and $10^{-5}$ (red), showing a close match in the region where $|\fads|$ is large.}
     \label{fig:Xcomp}
    \end{center}
  \end{figure}

\section*{Conclusion}

  In this paper, we have studied complex-valued autoconvolution problems with continuous kernel-functions in different data situations, arising, for example, in the characterization of ultrashort laser pulses by means of the SD-SPIDER method. We have derived fundamental analytical properties, in particular, weak-to-weak continuity in $L_\IC^2$-spaces of the autoconvolution operator, which ensure well-posed\-ness of regularization approaches by minimizing Tikhonov-type functionals either in $L_\IC^2$ or in compactly embedded subspaces, depending on the data at hand. Inspired by the TIGRA method and using discretizations in terms of NURBS curves, we have also proposed a novel globalized numerical method for phase retrieval problems corresponding to the real-world data situation, which is able to find solutions that appropriately reproduce the given data.

  Nevertheless, a number of open questions remain. Convergence rates results are, to be best of our knowledge, completely unavailable except for very special cases, and we have proven that classical nonlinearity conditions such as the tangential cone condition are not suitable to tackle autoconvolution problems. Concerning the noise model, further improvements might be possible with discrepancy terms that capture more adequately the multiplicative noise structure on the complex unit sphere for measurements of the complex phase function. Finally, we expect that an even better data fit could be achieved numerically by devising a method for blind deautoconvolution, i.e., by including the kernel function as a free or parameter-dependent variable in the optimization method.

\section*{Acknowledgement}

The first three authors greatly appreciate financial support by the German Research Foundation (DFG) under grants HO~1454/9-1 (SA and BH) and FL~832/1-2 (SB), respectively.

\section*{Appendix}

\subsection*{Proof of Proposition \ref{pro:weaklycont}}
We consider a sequence $\{f_n\}_{n \in \N} \in X$ with $f_n \rightharpoonup f_0$ in $X$ as $n \to \infty$ and show that $ \lim \limits_{n \to \infty}\langle F(f_n)-F(f_0),\eta\rangle_Y=0$ for all $\eta \in Y$, which proves the proposition.
Indeed, by Fubini's theorem we have
  \begin{align*}
    \langle F(f_n) & - F(f_0), \eta \rangle_Y\\
	& = \int \limits _0^2\int \limits_{\max(s-1,0)}^{\min(s,1)} k(s,\tau)(f_n(s-\tau)f_n(\tau)-f_0(s-\tau)f_0(\tau))\de \tau\, \eta(s) \de s\\
	& = \int \limits_0^2 \left[\int\limits_{\max(s-1,0)}^{\min(s,1)}  k(s,\tau)(f_n(s-\tau)-f_0(s-\tau))(f_n(\tau)+f_0(\tau))\de \tau \right]\, \eta(s) \de s
  \end{align*}
  \begin{align*}
	&  = \int \limits_0^1 (f_n(\tau)+f_0(\tau))\left[\int_\tau^{\tau+1} k(s,\tau)(f_n(s-\tau)-f_0(s-\tau))\eta(s) \de s \right] \de \tau\\
	&  = \int \limits_0^1 (f_n(\tau)+f_0(\tau))\left[\int \limits_0^{1} k(\xi+\tau,\tau)(f_n(\xi)-f_0(\xi))\eta(\xi+\tau) \de \xi\right] \de \tau.
  \end{align*}
If we use the settings $\Delta_n(\xi):=f_n(\xi)-f_0(\xi)$, $\Theta_\tau(\xi):=k(\xi+\tau,\tau)\eta(\xi+\tau)$ and $\Xi_n(\tau):=\int_0^1 \Theta_\tau(\xi) \Delta_n(\xi) \de \xi$, where $\Theta_\tau \in X$ for all $\tau \in [0,1]$
due to the continuity of the kernel $k$
and $\Delta_n,\Xi_n \in X$ for all $n \in \N$, we arrive at
\[\langle F(f_n)-F(f_0),\eta\rangle_Y =\int\limits_0^1 (f_n(\tau)+f_0(\tau)) \Xi_n(\tau) \de \tau = \langle f_n + f_0 , \Xi_n \rangle_X.\]
By some calculations it can be shown that the family of functions $\Xi_n,\;n \in \N,$ is equicontinuous on the interval $[0,1]$. On that interval, the sequence $\{\Xi_n\}_{n \in \N}$ converges  pointwise to zero, because
$\Delta_n \rightharpoonup 0$ in $X$ implies
$$ \Xi_n(\tau)= \langle \Theta_\tau,\Delta_n \rangle_X \to 0 \quad \mbox{as} \quad n \to \infty \quad \mbox{for all} \quad \tau \in [0,1].$$
However, an equicontinuous and pointwise convergent sequence of functions is even uniformly convergent, which yields $\lim_{n\to \infty}\|\Xi_n\|_X=0.$
Then the limit condition
\[\left|\langle F(f_n)-F(f_0),\eta\rangle_Y \right| = \left|\langle f_n + f_0 , \Xi_n \rangle_X \right| \le \|f_n+f_0\|_X \, ||\Xi_n||_X \to 0\quad \mbox{as} \quad n \to \infty, \]
which is based on the fact that $\|f_n+f_0\|_X$ is bounded, completes the proof.
\hspace*{\fill} \qed \vspace{5mm}

\subsection*{Proof of Proposition \ref{pro:noncompact}}

In the special case $k\equiv 1$ on $\suppk$ from (\ref{eq:suppk}), the proof would be based on the fact that we have $F(h_n)\not \to 0$ for $h_n(t):=\expfct{int} \rightharpoonup 0$. For a general kernel function $k$, however, we have to take into account some more details. Since $k$ is not identically zero, there is some $(s_0,\tau_0) \in {\rm int}(\suppk)$ with $\underline{k}:=k(s_0,\tau_0)\ne0$. Due to the continuity of the kernel function $k$ on $\suppk$ there exists an open neighborhood $\mathcal{U}_0$ of $(s_0,\tau_0)$ with
$|k(s,\tau)-\underline{k}|\le \frac{|\underline{k}|}2$ for all $(s,\tau)\in \mathcal{U}_0 \cap {\rm int}(\suppk)$. If we define $D:=\{(2\tau,\tau):\,0 \le \tau \le 1\}$, then there exists $(s_1,\tau_1)$ and $\varepsilon>0$ such that with
\[\mathcal{U}_1:=\{(s,\tau)\in {\rm int}(\suppk):\, \tau_1 \le \tau \le \tau_1 + \varepsilon,\, s_1-\tau_1\le s-\tau \le s_1-\tau_1 + \varepsilon  \}  \]
we have $\mathcal{U}_1\subset \mathcal{U}_0 \cap {\rm int}(\suppk)$ and $\mathcal{U}_1\cap D=\emptyset$. It is not hard to show that
\begin{equation}\label{eq:lowint}
 \left|\int k(s,\tau) h(s-\tau) h(\tau) \diff \tau\right|\ge \left| \underline{k} \right| \left|\int h(s-\tau) h(\tau) \diff \tau \right|  - \frac{|\underline{k}|}2 \int \left| f(s-\tau) f(\tau) \right|\diff \tau
\end{equation}
for $h\in L^2_\IC(0,1)$ and $(s,\tau)\in \mathcal{U}_1$.
Now we define sequences $\{h_n\}_{n\in\IN}$ and $\{f_n\}_{n\in\IN}$ by
  \[h_n(\tau):=\begin{cases} \expfct{ i n \left(\frac{\tau-\tau_1}\varepsilon\right)} & \text{ for } \tau_1 \le \tau \le \tau_1 + \varepsilon\\ \expfct{ i n \left(\frac{\tau-s_1+\tau_1}\varepsilon\right)} & \text{ for } s_1-\tau_1\le t \le s_1-\tau_1+\varepsilon \\ 0 & \text{ else}   \end{cases} \qquad  f_n:=f+r\,h_n.\]
Obviously $h_n\rightharpoonup0,\;f_n\rightharpoonup f$ in $X$, and due to the weak continuity of $F$ (cf.~Proposition~\ref{pro:weaklycont}) $F(f_n)\rightharpoonup F(f)$ in $Y$.
With (\ref{eq:lowint}) at hand it is easy to show that
\[\liminf_{n\to \infty} ||F(h_n)||_Y>0,\]
which means that $F(h_n)\not \to 0$ in $Y$. Taking into account that the Fr\'echet derivative is compact and therefore completely continuous we have $F^\prime(f)\,h_n \to 0$ in $Y$. Together with formula (\ref{eq:equa}), we then obtain
 $$ F(f_n) - F(f)= r^2\,F(h_n) + r\, F'(f)\,h_n \quad \mbox{and thus} \quad F(f_n)\not \to F(f).$$
 The compactness of $F^\prime(f)$ is an immediate consequence of the fact that (\ref{eq:Frech}) is a linear Fredholm integral operator with square integrable kernel. Such operators are always Hilbert-Schmidt operators and thus compact. The nonlinear operator $F$, however, is not compact, since the sequence $\{f_n\}_{n \in \N} \subset X$ is bounded and weakly convergent to $f$ but the associated sequence $\{F(f_n)\}_{n \in \N} \subset Y$
 cannot have a convergent subsequence. This completes the proof.
\hspace*{\fill} \qed \vspace{5mm}

\subsection*{Proof of Proposition \ref{pro:tikhrelmin}}
     Let $\{ f_n \} \subset X_0$ be a sequence such that
      \[ \mathcal{T}_{\varepsilon, \alpha, \beta}^{\sigma, \delta} (f_n) \to \inf_{f \in X_0} \mathcal{T}_{\varepsilon, \alpha, \beta}^{\sigma, \delta} (f) < \infty. \]
     Passing to a subsequence if necessary, we may assume that $\|F(f_n)\|_Y \geq \varepsilon$ holds for all $n$. Due to the continuity of $F(f)$, the pre-image of the closed set $\{ \|y\|_Y \geq \varepsilon \}$,
	\[ A_\varepsilon := \{ f \in X \st \|F(f)\|_Y \geq \varepsilon \} \]
     is a closed subset of $X = L_\IC^2 (0,1)$. Now any accumulation point $\bar f$ of $\{ f_n \}$ with respect to the weak topology in $X_0$ ($\bar f$ exists due to the stabilizing properties of $\Reg(f)$) is a strong accumulation point of $\{ f_n \}$ in $X$ and hence contained in $A_\varepsilon$. Taking a subsequence, again denoted by $\{ f_n \}$, such that $f_n \wto \bar f$ in $X_0$ we thus have $f_n \to \bar f$ in $X$ and $F(f_n) \to F(\bar f)$ in $Y$ with $\| F(\bar f) \|_Y \geq \varepsilon$. This yields
      \[ \frac{F(f_n)  - |F(f_n)| \expfct{i \psinoisy}}{\| F(f_n) \|_Y} \to \frac{F(\bar f)  - |F(\bar f)| \expfct{i \psinoisy}}{\| F(\bar f) \|_Y} \mbox{ in } Y \qquad \mbox{and} \qquad |f_n| \to |\bar f| \mbox{ in } X. \]
     In combination with the weak lower semicontinuity of $\Reg(f)$ in $X_0$ we therefore obtain
      \begin{align*}        
	\mathcal{T}_{\varepsilon, \alpha, \beta}^{\sigma, \delta} (\bar f) \leq \liminf_{n \to \infty} \mathcal{T}_{\varepsilon, \alpha, \beta}^{\sigma, \delta} (f_n) = \inf_{f \in X_0} \mathcal{T}_{\varepsilon, \alpha, \beta}^{\sigma, \delta} (f)
      \end{align*}
     and the proof is complete.
\hspace*{\fill} \qed

\end{document}